\documentclass[onefignum,onetabnum,final]{siamart190516}

\usepackage{lipsum}
\usepackage{amsfonts}
\usepackage{graphicx}
\usepackage{epstopdf}
\usepackage{algorithmic}
\ifpdf
  \DeclareGraphicsExtensions{.eps,.pdf,.png,.jpg}
\else
  \DeclareGraphicsExtensions{.eps}
\fi

\newsiamremark{remark}{Remark}
\newsiamremark{hypothesis}{Hypothesis}
\crefname{hypothesis}{Hypothesis}{Hypotheses}
\newsiamthm{claim}{Claim}

\headers{Non-null-controllability of fractional heat equation equation}{A. 
Koenig}

\title{Lack of null-controllability for the fractional heat equation and related 
equations\thanks{This work was partially supported by the ERC 
advanced grant SCAPDE, seventh framework program, agreement no.~320845. This work is 
partially supported by a public grant overseen by the French National Research 
Agency (ANR) as part of the ``Investissements d'Avenir'''s program of the Idex PSL 
reference «ANR-10-IDEX-0001-02 PSL».}}

\author{Armand Koenig\thanks{Université Paris-Dauphine, Université PSL, CNRS, 
CEREMADE, 75016 PARIS, FRANCE.
  (\email{armand.koenig@dauphine.psl.eu}, \url{https://koenig.perso.math.cnrs.fr}).}}

\usepackage{amsopn}

\usepackage[english]{babel}
\usepackage[T1]{fontenc}
\usepackage{lmodern}
\usepackage{csquotes}
\usepackage{geometry}
\usepackage{amsmath, amssymb, mathtools}
\usepackage{tikz}
\usetikzlibrary{calc, positioning, intersections, arrows.meta, decorations.markings}
\usepackage[titletoc,title]{appendix}
\usepackage{dsfont}
\usepackage{microtype}
\usepackage{hyperref}

\DeclareMathOperator{\supp}{support}
\DeclareMathOperator{\sign}{sgn}
\DeclareMathOperator{\atan}{arctan}
\DeclareMathOperator{\Ai}{Ai}

\newcommand{\diff}[1][-3]{\ensuremath{\mathop{}\mkern#1mu\mathrm{d}}}
\newcommand{\set}{\mathbb}
\newcommand{\per}[1]{#1\mkern0mu_{\mathrm{per}}}
\newcommand{\ssqrt}[1]{\sqrt{\smash[b]{#1}}}
\newcommand{\cross}[1]{+(-#1,0) -- +(#1,0) +(0,-#1) -- +(0,#1)}
\newcommand{\s}{\mkern1.5mu}
\newcommand{\Cc}{{\mathcal C}}
\newcommand{\puncfootnote}[1]{\kern-0.2em\footnote{#1}}
\newcommand{\smallo}{o}
\newcommand{\bigO}{{\mathcal O}}
\newcommand{\cri}{{\mathrm c}}
\newcommand{\csubset}{\Subset}
\allowdisplaybreaks[1]

\newcommand{\step}[1]{\vskip-\lastskip\medskip\noindent\textit{#1}}

\ifpdf
\hypersetup{
  pdftitle={Non-null-controllability of the fractional heat equation}{A. 
Koenig},
  pdfauthor={A. Koenig}
}
\fi

\newcommand{\Ighv}{I_{\gamma,h,v}(x)}
\begin{document}

\maketitle

\begin{abstract}
  We consider the equation $(\partial_t + 
\rho(\sqrt{-\Delta}))f(t,x) = \mathds 1_\omega u(t,x)$, $x\in \set R$ or $\set T$. We 
prove it is not null-controllable if $\rho$ is analytic on a conic neighborhood 
of $\set R_+$ and $\rho(\xi) = \smallo(|\xi|)$. The proof relies essentially on 
geometric optics, i.e.\ estimates for the evolution of semiclassical coherent states.

The method also applies to other equations. The most interesting example might be the 
Kolmogorov-type equation $(\partial_t -\partial_v^2 + 
v^2\partial_x)f(t,x,v) = \mathds 1_\omega u(t,x,v)$ for $(x,v)\in \Omega_x\times 
\Omega_v$ with $\Omega_x = \set R$ or $\set T$ and $\Omega_v = \set R$ or $(-1,1)$. 
We prove it is not null-controllable in any time 
if $\omega$ is a vertical band $\omega_x\times \Omega_v$. The idea is to remark 
that, for some families of solutions, the Kolmogorov equation behaves like 
the rotated fractional heat equation $(\partial_t + 
\sqrt i(-\Delta)^{1/4})g(t,x) = \mathds 1_\omega u(t,x)$, $x\in \set T$.
\end{abstract}



\section{Introduction}
\subsection{Problem of the null-controllability}
Consider the following equation, which is called the \emph{fractional heat 
equation}, where $\Omega = \set R$ or $\set T$, $\omega$ is an open 
subset of $\Omega$, $\alpha\geq 0$:
\begin{equation*}
  (\partial_t + (-\Delta)^{\alpha/2})f(t,x) = \mathds 1_\omega u(t,x) 
\quad t\in [0,T], x \in \Omega
\end{equation*}
Here, we define 
$(-\Delta)^{\alpha/2}$ with the functional calculus, that is,
$(-\Delta)^{\alpha/2}f = \mathcal F^{-1}(|\xi|^{\alpha}\mathcal F(f))$ if 
$\Omega = \set R$, where $\mathcal F$ is the Fourier transform; and 
$c_n((-\Delta)^{-\alpha/2} f) = |n|^\alpha c_n(f)$ if $\Omega = \set T$, where 
$c_n(f)$ is the $n$th Fourier coefficient of $f$. 

It is a control problem with state $f\in L^2(\Omega)$ and control $u$ supported in 
$\omega$. More precisely, we are interested in the exact null-controllability of 
this equation.
\begin{definition}\label{def:control_rfhe}
We say that the fractional heat equation is null-controllable on $\omega$ in 
time $T>0$ if for all $f_0$ in $L^2(\Omega)$, there exists $u$ in $L^2([0,T]\times 
\omega)$ 
such that the solution $f$ of:
\begin{equation}\label{eq:rfhe_control}
 \begin{aligned}
  (\partial_t + (-\Delta)^{\alpha/2})f(t,x) &= \mathds 1_\omega u(t,x) 
&&\quad t\in [0,T], x \in \Omega\\
  f(0,x) &= f_0(x) &&\quad x \in \Omega.
 \end{aligned}
\end{equation}
satisfies $f(T,x,v) = 0$ for all $(x,v)$ in $\Omega$.
\end{definition}

The main motivation for this study, apart from studying the fractional heat 
equation itself, is the null-controllability of a Kolmogorov-type equation. More 
specifically, we are interested in the following equation,
where $\Omega = \Omega_x \times \Omega_v$ with $\Omega_x = \set R$ or $\set 
T$, $\Omega_v = \set R$ or $(-1,1)$ and $\omega$ is an open subset of $\Omega$:
\begin{equation*}
  (\partial_t + v^2 \partial_x - \partial_v^2)f(t,x,v) = \mathds 1_\omega 
u(t,x,v) \quad t\in [0,T], (x,v) \in \Omega.
\end{equation*}
For convenience, we will just say in this paper ``the Kolmogorov equation''.
Note that thanks to Hörmander's bracket condition~\cite[Section 
22.2]{hormander_2007}, 
the operator $v^2\partial_x - \partial_v^2$ is hypoelliptic. Also, this equation is 
well-posed. This can be proved by Hille-Yosida's theorem (see \cite[Section 
4]{beauchard_2014a} in the case $\Omega = \set T \times  (-1,1)$).
As we will see, this Kolmogorov equation is related to the rotated fractional 
heat equation.

\begin{definition}\label{def:control_kolm}
We say that the Kolmogorov equation is null-controllable on $\omega$ in time $T>0$ 
if for all $f_0$ in $L^2(\Omega)$, there exists $u$ in $L^2([0,T]\times \omega)$ 
such that the solution $f$ of:
\begin{equation}\label{eq:kolm_control}
 \begin{aligned}
  (\partial_t + v^2 \partial_x - \partial_v^2)f(t,x,v) &= \mathds 1_\omega 
u(t,x,v) &&\quad t\in [0,T], (x,v) \in \Omega\\
  f(0,x,v) &= f_0(x,v) &&\quad (x,v) \in \Omega\\
  f(t,x,v) &= 0&&\quad t\in[0,T],(x,v)\in \partial \Omega \text{ (if non-empty)}
 \end{aligned}
\end{equation}
satisfies $f(T,x,v) = 0$ for all $(x,v)$ in $\Omega$.
\end{definition}

\subsection{Statement of the results}
We will prove that the rotated fractional heat equation is never null controllable 
if $\Omega\setminus\omega$ has nonempty interior, and that the Kolmogorov equation 
is never null-controllable if $\omega = \omega_x\times \Omega_v$ where 
$\Omega_x\setminus \omega_x$ has nonempty interior.

\begin{theorem}\label{th:th_heat}
 Let $0\leq\alpha<1$ and $\Omega = \set R$ or $\Omega = 
\set T$. Let $\omega$ be a strict open subset of $\Omega$. The fractional 
heat equation \eqref{eq:rfhe_control} is not null controllable in any time on 
$\omega$.
\end{theorem}

This Theorem still holds in higher dimension, with $\Omega = \set R^d\times 
\set T^{d'}$, but our method seems ineffective to treat the case where $\Omega$ is, 
say, an open subset of $\set R$. 
This may be because we are using the spectral definition of the fractional 
Laplacian, and our method might be adapted if instead we used a singular kernel 
definition 
of the fractional Laplacian. 

Actually, we prove the non-null controllability of a class of equations of the form 
$(\partial_t + \rho(\ssqrt{-\Delta}))f(t,x) = \mathds 1_\omega(t,x)$.
\begin{theorem}\label{th:gen_frac}
 Let $K>0$, $\Cc = \{\xi \in \set C, \Re(\xi)>K, |\Im(\xi)| 
<K^{-1}\Re(\xi)\}$ and $\rho\colon \Cc\cup \set R_+ \to \set C$ such that
\begin{enumerate}
 \item $\rho$ is holomorphic on $\Cc$,
 \item $\rho(\xi) = \smallo(|\xi|)$ in the limit $|\xi|\to +\infty$, $\xi\in \Cc$,
 \item $\rho$ is measurable on $\set R_+$ and $\inf_{\xi\in \set R_+} \Re(\rho(\xi)) 
> -\infty$.
\end{enumerate}

Let $\Omega = \set R$ or $\set T$, $\omega$ be a strict open subset of $\Omega$ 
and $T>0$. Then the equation
\begin{equation}
 \label{eq:gen_frac_control}
 (\partial_t + \rho(\ssqrt{-\Delta}))f(t,x) = \mathds 1_\omega u(t,x),\quad 
t\in[0,T],\ x\in \Omega
\end{equation}
is not null-controllable on $\omega$ in time $T$.
\end{theorem}

For lack of a better name, we will call the equation~\eqref{eq:gen_frac_control} the 
\emph{generalized fractional heat equation}. This Theorem can be generalized to the 
case $\Omega = \set R^d\times \set T^{d'}$. The hypothesis $\inf_{\set R_+} \Re(\rho) 
> -\infty$ is only used to ensure that the equation is well-posed.

The fractional heat equation is the case $\rho(\xi) = \xi^\alpha$. Note that if 
$\alpha=0$, then the fractional heat equation is just a 
family of decoupled ordinary differential equations, and the conclusion of 
\cref{th:th_heat} is unimpressive. At the other end, the method used in this article 
does not work as-is if $\alpha=1$, but we still expect non-null-controllability, even 
if this remains a conjecture if $\Omega$ is not the one-dimensional torus.

Some equations behave like the fractional heat equation, at least in some regimes. 
This is the case of the Kolmogorov equation, and if the control acts on a vertical 
band, we will prove it is not null-controllable with the same method.
\begin{theorem}\label{th:th_kolm}
Let $\Omega_x = \set R$ or $\set T$, let $\Omega_v = \set R$ or $(-1,1)$, and let 
$\Omega = \Omega_x \times \Omega_v$. Let $T>0$ and $\omega_x$ be a strict open subset 
of $\Omega_x$. The Kolmogorov equation~\eqref{eq:kolm_control} is not 
null-controllable on $\omega = \omega_x \times \Omega_v$ in time $T$.
\end{theorem}

This Theorem can be extended to higher dimension in $x$ and $v$ if 
$\Omega_v = \set R^d$. If we want, say $\Omega_v = (-1,1)^d$, we lack information on 
the eigenvalues and eigenfunctions of $-\partial_v^2 + in v^2$ on $(-1,1)^d$, but 
this is the only obstacle to the generalization of the Theorem to this case. 
We also 
give a non-null-controllability result in small time for more general control region.

\begin{theorem}\label{th:th_kolm_general}
 Let $\Omega_x = \set R$ or $\Omega_x = \set T$ and $\Omega_v = \set R$ or $\Omega_v 
= (-1,1)$. Let $\Omega = \Omega_x \times \Omega_v$. Let $\omega 
\subset \Omega$. Assume that there exist $x_0\in \Omega_x$ and 
$a>0$ such that the symmetric vertical interval $\{(x_0,v), -a<v<a\}$ is disjoint 
from $\overline \omega$. Then, the Kolmogorov equation~\eqref{eq:kolm_control} is not 
null-controllable on $\omega$ in time $T<a^2\!/2$. 
\end{theorem}

Whether this condition $T<a^2\!/2$ is optimal or not is an open question, but we 
conjecture that it is optimal, at least for some geometries. If $\omega = \set 
T\times (a,b)$ with $0<a<b$, \cref{th:th_kolm_general} proves that 
null-controllability does not hold in time $T< a^2\!/2$. This special case was 
already known~\cite[Theorem 1.3]{beauchard_2015a}, it is also proved in the same reference 
that null-controllability holds for some $T>0$. Our \cref{th:th_kolm_general} 
sharpens the lower-bound on the minimal time of null-controllability if the geometry 
of $\omega$ is different than a cartesian product.

While the fractional heat equation and the Kolmogorov equation are the main focus of 
this article, the method can be used to treat other equations: those that 
behave like the fractional heat equation for $\alpha<1$. In \cref{sec:other_eq}, we 
briefly discuss the fractional Schrödinger equation, and
sketch the proof for the Kolmogorov-type equation $(\partial_t -\partial_v^2 
-v\partial_x)f(t,x,v) = \mathds 1_\omega u(t,x,v)$ (notice the $v$ instead of the 
$v^2$), and the \emph{improved Boussinesq equation} $(\partial_t^2 -\partial_x^2 
-\partial_x^2\partial_t^2)f(t,x) = \mathds 1_\omega u(t,x)$.

\subsection{Bibliographical comments}
\subsubsection{Control of partial differential equations}\label{sec:contr_abst}
Let $A$ be on operator on a Hilbert space $H$ such that the equation $\partial_t f + 
Af = 0$ is well-posed (i.e.\ $-A$ generates a strongly continuous semigroup of 
bounded linear operators on $H$, see for instance~\cite[Ch.~2]{tucsnak_2009} 
or~\cite[Sec.~2.3 and Appendix~A]{coron_2007} for the 
definition).

Let $U$ be a Hilbert space and $B\colon U\to H$ a bounded linear operator. With 
the right choice of $H$, $A$, $U$ and $B$, the problems we are interested in can be 
stated the following way: for every $f_0\in H$, does there exist $u\in L^2(0,T;U)$ 
such that the solution of $\partial_t f +A f = Bu$, $f(0) = f_0$ satisfies $f(T) = 
0$?

For the fractional heat equation~\eqref{eq:rfhe_control} on $\set R^n$, we 
choose $H = L^2(\set R^n)$, $A = (-\Delta)^{\alpha/2}$ with domain $H^\alpha(\set 
R^n)$, $U = L^2(\omega)$ and $B\colon u\mapsto u\mathds 1_{\omega}$. For the 
Kolmogorov equation~\eqref{eq:kolm_control} on $\set R^2$, we choose $H = L^2(\set 
R^2)$, $A = -\partial_v^2 + v^2\partial_x$ (the domain of $A$ is a bit complicated to 
define, see~\cite[Sec.~4]{beauchard_2014a}), $U = L^2(\omega)$ and $B\colon u\mapsto 
u\mathds 1_\omega$.

Whether there exists a $u\in L^2(0,T;U)$ such that the solution of $\partial_tf +Af = 
Bu$, $f(0) = f_0$ satisfies $f(T) = 0$ depends of course of $A$, $B$ and on the 
spaces $H$ and $U$. Let us discuss existing results when $A$ is a parabolic operator 
related to the fractional heat equation or the Kolmogorov equation.

\subsubsection{Null-controllability of the (fractional) heat equation in dimension 
one: the moment method}
Let us first look at the heat equation in dimension one with Dirichlet boundary 
conditions, i.e.\ $H = L^2(0,\pi)$, $D(A) = H^2(0,\pi)\cap H^1_0(0,\pi)$ and $A\colon 
f\in D(A)\mapsto -\partial_x^2 f \in L^2(0,\pi)$. Let us also denote $\lambda_n$ the 
eigenvalues of $A$, and assume that $\lambda_n$ is increasing, so that $\lambda_n = 
n^2$.

A possible strategy to control the heat equation in dimension one is to look for a 
control of the form $\tilde u(t,x) = u(t)v(x)$. In the framework of 
\cref{sec:contr_abst}, this is 
the choice $U = \set R$ and $Bu$ the function $x\in(0,\pi)\mapsto uv(x)$. We 
will call this kind of controls \emph{shaped controls}. This is the strategy pioneered 
by Fattorini and Russel~\cite{fattorini_1971}. Let us describe it briefly.

Let $v\colon (0,\pi)\to \set R$. Let $f_0\in L^2(0,\pi)$, 
let $u\in L^2(0,T)$ and let $f$ be the solution of $(\partial_t -\partial_x^2)f(t,x) 
= u(t)v(x)$, $f(t,0) = f(t,\pi) = 0$ with initial condition $f(0,x) = f_0(x)$. 
Finally, for every $g\in L^2(0,\pi)$, let $c_n(g) \coloneqq \int_0^\pi 
g(x)\sin(nx)\diff x$ be the $n$-th Fourier coefficient of $g$. Then, the relation 
$f(T,\cdot) = 0$ is equivalent to the \emph{moment problem}
\[
 \forall n\in \set N\setminus\{0\},\ \int_0^T e^{(t-T)\lambda_n}u(t)\diff t = 
-\dfrac{e^{-T\lambda_n}c_n(f_0)}{c_n(v)}.  
\]

Fattorini and Russel prove such a $u$ exists by constructing a biorthogonal family to 
$(e^{-t\lambda_n})_{n\in\set N\setminus\{0\}}$, i.e.\ a family of functions 
$(g_n)_{n\in\set N\setminus\{0\}}$ such that $\int_0^T g_n(t)e^{-\lambda_mt} \diff t 
= 1$ if $n = m$ and $0$ if $n\neq m$ (see also~\cite{tenenbaum_2007} for a more 
streamlined proof that this family exists). Then the function $u$ defined by
\[
 u(t) = -\sum_{n>0} \dfrac{e^{-T\lambda_n}c_n(f_0)}{c_n(v)} g_n(T-t)
\]
formally solves the moment problem. Moreover, we can prove some estimates on the 
functions $(g_n)_{n>0}$, and if $c_n(v)$ does not decay too fast when $n\to +\infty$, 
the series that defines $u$ actually converges.

This strategy can be adapted for the fractional heat equation~\eqref{eq:rfhe_control} 
when $\alpha>1$, as Micu and Zuazua~\cite{micu_2006} already remarked. Indeed, the 
construction and estimate on the biorthogonal family relies on the hypotheses 
$\sum_{n>0} |\lambda_n|^{-1}<+\infty$ and $\lambda_{n+1}-\lambda_n\geq c>0$. These 
hypotheses still hold if we replace the operator $A = -\partial_x^2$ on $(0,\pi)$ by 
$(-\partial_x^2)^{\alpha/2}$ as long as $\alpha>1$. Indeed, the eigenvalues are 
now $\lambda_n = n^\alpha$.

On the other hand, if $\alpha\leq 1$, this proof does not work anymore. In fact, Micu 
and Zuazua~\cite[Sec.~5]{micu_2006} proved that if $\alpha\leq 1$, the fractional 
heat equation~\eqref{eq:rfhe_control} is not null-controllable with controls of the 
form $u(t)v(x)$. Miller~\cite[Sec.~3.3]{miller_2006} (see 
also~\cite[Appendix]{duyckaerts_2012}) also gets similar 
results, with similar methods.

But these negative results, based on Müntz Theorem, only concern \emph{scalar 
controls}, i.e.\ the case where the control space is $U = \set R$ (or $\set C$). If 
the control space is larger, say $U = L^2(\omega)$, we cannot rule out the existence 
of a control with Müntz Theorem.  Indeed, there are many equations which are not 
null-controllable with scalar controls, but that are null-controllable with a larger 
control space. One of them is the heat equation in dimension larger than one. Let us 
discuss it now.


\subsubsection{Null-controllability of the (fractional) heat equation and the 
spectral inequality}

Let $\Omega$ be a bounded open subset of $\set R^d$. Let $(\lambda_n)_{n\geq 0}$ be 
the sequence of the eigenvalues of $-\Delta$ on $\Omega$ with Dirichlet boundary 
conditions.\footnote{I.e.\ $D(A) = \{f\in H^1_0(\Omega),\ \Delta f\in L^2(\Omega)\}$ 
where $\Delta f$ is in the sense of distributions.} According to Weyl's law, if 
the dimension $d$ is greater than $1$, then $\sum_{n\geq 0} \lambda_n^{-1} = 
+\infty$, so we cannot prove null-controllability with the moment method. To 
build a control for the heat equation we have to choose a control space that is 
infinite dimensional~\cite[Th.~2.79]{coron_2007}. We choose $U = L^2(\omega)$ and 
$B\colon u\in L^2(\omega) \mapsto u\mathds 1_\omega \in L^2(\Omega)$. We call this 
kind of controls \emph{internal controls}.

To prove the null-controllability of the heat equation in any dimension, Fursikov and 
Immananuvilov~\cite{fursikov_1996} use \emph{parabolic Carleman inequalities}, which 
are weighted energy estimates, to prove (more or less) directly the observability 
inequality that is equivalent to the null-controllability (see for 
instance~\cite[Th.~2.44 or Th.~2.66]{coron_2007} for the equivalence between 
null-controllability and observability).

Independently, Lebeau and Robbiano~\cite{lebeau_1995, lerousseau_2012} developped 
another strategy to prove the null-controllability of the heat equation. This 
strategy yields more insight for our purpose, so let us give more details.

By means of \emph{elliptic Carleman inequalities}, Lebeau and Robbiano proved a
\emph{spectral inequality}, which is the 
following: let $M$ be a connected compact riemannian manifold with boundary, let 
$\omega$ be an 
open subset of $M$, and let $(\phi_i)_{i\in \set N}$ be an orthonormal basis of 
eigenfunctions of $-\Delta$ with associated eigenvalues $(\lambda_i)_{i\in \set N}$, 
then there exists $C>0$ and $K>0$ such that for every sequence of complex numbers 
$(a_i)_{i\in \set N}$ and every $\mu >0$
\begin{equation}
 \Big|\!\sum_{\lambda_i < \mu}a_i \phi_i\Big|_{L^2(M)} \leq 
 Ce^{K\sqrt{\mu}} \Big|\!\sum_{\lambda_i < \mu} a_i \phi_i\Big|_{L^2(\omega)}.
\end{equation}

The key point to deduce the null-controllability of the heat equation from this 
spectral inequality is that if one takes an initial condition of the form $f_0 = 
\sum_{\lambda_i\geq \mu} a_i \phi_i$ with no component along frequencies less than 
$\mu$, the solution of the heat equation decays like $e^{-T\mu}|f_0|_{L^2(M)}$, and 
the exponent in $\mu$ in this decay (i.e.\ $1$) is larger than the one appearing in 
the spectral inequality (i.e.\ $1/2$).\footnote{The exponent $\frac12$ of 
$\mu$ in the spectral inequality is optimal if $\omega$ is a strict open subset of 
$M$~\cite[Proposition 5.5]{lerousseau_2012}.}

For the fractional heat equation~\eqref{eq:rfhe_control}, the dissipation stays 
stronger that the spectral inequality as long as $\alpha>1$. Thus, for $\alpha>1$, we 
can prove the null-controllability with Lebeau and Robbiano's method, as already 
mentioned by Micu and Zuazua~\cite{micu_2006} and Miller~\cite{miller_2006} (see 
also~\cite{miller_2010, duyckaerts_2012}).

Our Theorem~\ref{th:th_heat} proves that the threshold $\alpha>1$ is optimal: if 
$\alpha<1$, then the fractional heat equation is not null-controllable (at least for 
$\Omega = \set T^n$). Note that the case $\alpha = 1$ and $\Omega = \set T$ has 
already been proved to lack null-controllability with internal 
controls~\cite[Th.~4]{koenig_2017}. Let us also mention an article where the 
null-controllability of an equation closely related to our fractional heat equation 
have been investigated~\cite{biccari_2019}.

So it seems that the Lebeau-Robbiano method is in some sense optimal: if the 
dissipation is not stronger than the spectral inequality, then we do not 
have null-controllability. Let us finish with another class of parabolic 
equations for which Lebeau and Robbiano's method does not work: degenerate parabolic 
equations.

\subsubsection{Null-controllability of degenerate parabolic partial differential 
equations}\label{sec:kolm_bib}

Degenerate parabolic equations are equations of the form $\partial_t f(t,x) + A 
f(t,x) 
= \mathds 1_\omega u(t,x)$, $0\leq t\leq T$, $x\in \Omega$, where $A$ is a 
second-order differential operator which is degenerate elliptic, i.e.\ its principal 
symbol $P(x,\xi)$ satisfies $P(x,\xi) \geq 0$ but is zero for some $x\in \Omega$ and 
$\xi \neq 0$.

The interest in the null-controllability of degenerate parabolic equations is more 
recent. We now understand the null-controllability of parabolic equations degenerating 
at the boundary in dimension one~\cite{cannarsa_2008} and two~\cite{cannarsa_2016} 
(see also references therein), where the authors found that these equations where 
null-controllable if the degeneracy is not too strong, but might not be if the 
degeneracy is too strong. To the best of our knwoledge, the only other general family 
of degenerate parabolic equations whose null-controllability has been investigated 
is hypoelliptic quadratic differential equations~\cite{beauchard_2018b, 
beauchard_2017}.

Other equations have been studied on a case-by-case basis. For instance, the 
Kolmogorov equation has been investigated since 2009 
\cite{beauchard_2009, beauchard_2014a, beauchard_2015a}. In these papers, the authors 
found that if $\Omega = \set T\times (-1,1)$ and $\omega = \set T\times (a,b)$ with 
$0<a<b<1$, the Kolmogorov equation is null-controllable in 
large times, but not in time smaller than $a^2\!/2$, and that if $-1<a<0<b<1$, it is 
null-controllable in arbitrarily small time~\cite{beauchard_2014a}. If in the 
Kolmogorov equation~\eqref{eq:kolm_control} we replace $v^2$ by $v$, the 
null-controllability holds in arbitrarily small time if $\omega = \set T\times 
(a,b)$~\cite{beauchard_2009, beauchard_2014a}. On the other 
hand, if we replace $v^2$ by $v^\gamma$ where $\gamma$ is an integer greater than 
$2$ and $\omega = \set T \times (a,b)$, it is never 
null-controllable~\cite{beauchard_2015a}. In this last article, the 
null-controllability of a model of the equation we are interested in, namely the 
equation $(\partial_t + iv^2 (-\Delta_x)^{1/2} - \partial_v^2)g = 0$, is also 
investigated.

Another degenerate parabolic equation is the Grushin equation 
$(\partial_t-\partial_x^2 -x^2 \partial_y^2)f(t,x,y) = \mathds 1_\omega u(t,x,y)$ 
on $\Omega = (-1,1)\times (0,\pi)$ with Dirichlet boundary conditions. If the control 
domain is a vertical band $\omega = (a,b)\times (0,\pi)$ with $0<a<b$, there exists 
a minimum time for the null-controllability to hold~\cite{beauchard_2014}. This 
minimum time has since been computed~\cite{beauchard_2018}. On the other end, if 
the domain control is an horizontal band $\omega = (0,\pi)\times (a,b)$ with 
$(a,b)\subsetneq (0,\pi)$, then the Grushin equation is not null 
controllable~\cite{koenig_2017}. 

Let us finally just mention an article on the heat equation on 
the Heisenberg group since 2017~\cite{beauchard_2017a}, and that some parabolic 
equations on the real half-line, some of them related to the present work, have been 
shown to strongly lack controllability~\cite{darde_2018}.

\subsection{Outline of the proof, structure of the article}\label{sec:plan}
As usual in controllability problems, we focus on \emph{observability 
inequalities} on the \emph{adjoint systems}, that are equivalent to the
null-controllability (see \cite[Theorem 2.44]{coron_2007}). 

Specifically, the null-controllability of the fractional heat 
equation~\eqref{eq:rfhe_control} is equivalent to the existence of $C>0$ such that 
for every solution $g$ of
\begin{equation}\label{eq:rfhe_adj}
 (\partial_t + (-\Delta)^{\alpha/2})g(t,x) = 0 \quad t\in(0,T), x\in \Omega
\end{equation}
we have
\begin{equation}\label{eq:obs_rfhe}
 |g(T,\cdot)|_{L^2(\Omega)} \leq C|g|_{L^2((0,T)\times \omega)}.
\end{equation}
So, to disprove the null controllability, we only have to find solutions 
of~\eqref{eq:rfhe_adj} that are concentrated outside $\omega$. To construct such 
solutions, we consider initial states that are (essentially) 
\emph{semiclassical coherent states}, 
i.e.\ initial states of the form $g_{0,h}\colon x\mapsto h^{-1/4} 
e^{-(x-x_0)^2\!/2h +ix\xi_0/h}$. 
We will prove that solutions of Eq.~\eqref{eq:rfhe_adj} with these initial conditions 
stay concentrated around $x_0$. More precisely, we get asymptotic expansion of these 
solutions 
thanks to the saddle point method. We do this informally at first, in 
\cref{sec:informal}, then rigorously in \cref{sec:heat_R} in the case $\Omega = \set 
R$ and in \cref{sec:heat_bounded} in the case $\Omega = \set T$. This proof relies on 
some technical computations that are done in \cref{sec:technical}. These computations 
are carried over in a slightly general framework, that allows to directly treat the 
other equations, namely the Kolmogorov-type equations and the improved 
Boussinesq equation. We also sketch the proof of the generalization of 
\cref{th:gen_frac} in higher dimension in \cref{sec:gen_frac_higher_dim}.

Let us finish this introduction by explaining how the Kolmogorov equation for
$\Omega_x = \Omega_v = \set R$ and the fractional heat equation are related. The 
first eigenfunction of $-\partial_v^2 + i\xi v^2$ on $\set R$, is\footnote{Here and 
in 
all this paper, we choose the branch of the square root with positive real part.} 
$e^{-\ssqrt{i\xi}\s v^2\!/2}$ (up to a normalization 
constant), with eigenvalue $\sqrt{i\xi}$. So, $\Phi_\xi \colon(x,v) \in \set R^2 
\mapsto 
e^{i\xi x-\sqrt{i\xi}\s v^2\!/2}$ is a generalized eigenfunction of the Kolmogorov 
operator $v^2\partial_x - \partial_v^2$, with eigenvalue $\sqrt{i\xi}$. So, the 
solution of the Kolomogorov equation $(\partial_t + v^2\partial_x - \partial_v^2)f= 
0$ with initial condition $f(0,x,v) = \int_{\set R} a(\xi) \Phi_\xi(x,v)\diff\xi$ is 
$f(t,x,v) = \int_{\set R} a(\xi) \Phi_\xi(x,v)e^{-\ssqrt{i\xi}\s t}\diff\xi$. This 
suggests that, dropping the $v$ variable for the moment, the Kolmogorov equation 
behaves like an equation where the eigenfunctions are the $e^{i\xi x}$ with 
eigenvalue $\sqrt{i\xi}$, i.e.\ the equation $(\partial_t + \sqrt i 
(-\Delta_x)^{1/4})f(t,x) = 0$ with $x\in \set R$.

Based on this observation and the non-null-controllability result of the rotated 
fractional heat equation on the whole real line, we prove in \cref{sec:kolm_R} that 
the Kolmogorov equation is not null-controllable in the case $\Omega = 
\set R\times \set R$. For Kolmogorov's equation on $\Omega = \Omega_x\times 
(-1,1)$, we need some information on the eigenvalues and the eigenfunctions, which 
are not explicit anymore. We already proved most of what we need in 
another article \cite[Section 4]{koenig_2017}. We prove the non-null-controllability 
of Kolmogorov equation with $\Omega_v = (-1,1)$ in \cref{sec:kolm_bounded}.

Finally, we sketch the proof for the Kolmogorov equation with $v$ instead of $v^2$ 
and for the improved Boussinesq equation in \cref{sec:other_eq}.

\section{Informal presentation of the proof}\label{sec:informal}
As we explained in \cref{sec:plan}, we will try to disprove the observability 
inequality~\eqref{eq:obs_rfhe}. We only discuss here the case $\Omega = \set R$.

Since the fractional heat equation is invariant by 
translation, we may assume that $\omega \subset \{|x|>\delta\}$ for some $\delta>0$.
Then, for $h>0$, we consider the solution $g_h$ of the fractional heat 
equation~\eqref{eq:rfhe_adj} with initial condition $g_{0,h}(x) = 
e^{-x^2\!/2h+i\xi_0/h}$ with some $\xi_0>0$. The solution of the fractional heat 
equation is then
\[
 g_h(t,x) = \frac1{2\pi}\int_{\set R} e^{-t|\xi|^{\alpha}+ix\xi}\mathcal 
F(g_{0,h})(\xi)\diff \xi,
\]
where $\mathcal F$ is the Fourier transform defined by $\mathcal F f(\xi) = 
\int_{\set R} f(x) e^{-ix\xi} \diff x$.
But the Fourier transform of $g_{0,h}$ has a closed-form expression. Indeed, 
$\mathcal 
F(e^{-x^2\!/2})(\xi) =  \sqrt{2\pi} e^{-\xi^2\!/2}$, and using the properties of the 
Fourier transform (scaling and translation), we find $\mathcal F(g_{0,h})(\xi) = 
\sqrt{2\pi h}\s e^{-(h\xi-\xi_0)^2\!/2h}$. Thus
\[
 g_h(t,x) = \sqrt{\frac h{2\pi}}\int_{\set R} 
e^{-(h\xi-\xi_0)^2\!/2h +ix\xi-t|\xi|^{\alpha}}\mathcal\diff \xi.
\]
If we make the change of variables $\xi' = h\xi$, we find
\begin{equation}\label{eq:solution_simple}
 g_h(t,x) = \frac1{\sqrt{2\pi h}}\int_{\set R} 
e^{-(\xi-\xi_0)^2\!/2h +ix\xi/h-t|\xi|^{\alpha}\!/h^\alpha}\mathcal\diff \xi.
\end{equation}

Notice that if $h$ is small, the term $e^{-(\xi-\xi_0)^2\!/2h}$ is concentrated 
around $\xi = \xi_0$, and so is the integrand. 
Thus, the major part of this integral comes from a neighborhood of $\xi_0$. In this 
situation, we can compute asymptotic expansion with the saddle point method.

More precisely, the saddle point method (see for instance 
\cite[Ch.~2]{sjostrand_1982}) is a 
way to compute asymptotic expansion of integrals of the 
form $I(h) =\int e^{\phi(x)/h} u(x)\diff x$ in the limit $h\to 0^+$, where $\phi$ 
and $u$ are entire functions. 

If the main contribution in the integral $I(h)$ comes from a nondegenerate 
critical point of $\phi$ at $x=0$, the ``standard'' saddle point method gives
\begin{equation}
 I(h) \sim e^{\phi(0)/h}\sum_{\smash[b]{k}} \sqrt{2\pi}\,\tilde u_{2k} h^{k+1/2}
\end{equation}
where the $\tilde u_{2k}$ are of the form $A_{2k} u(0)$, and $A_{2k}$ are 
differential operators of order $2k$, with in particular $\tilde u_0 = u(0) 
|\phi''(0)|^{-1/2}$.

The term $e^{-|\xi|^\alpha\!/h^\alpha}$ prevents us from applying the saddle point 
method to Eq.~\eqref{eq:solution_simple} as-is, but let us pretend we can do it 
anyway (the rigorous computation will 
be carried in \cref{sec:technical}). Since the critical point of $\xi\mapsto 
-(\xi-\xi_0)^2\!/2+ix\xi$ is $\xi_\cri = \xi_0+ ix$, we get from the saddle point 
method
\[
 g_h(t,x) \approx  e^{-(\xi_\cri-\xi_0)^2\!/2h +ix\xi_c/h -t 
\xi_c^\alpha\!/h^\alpha} =  e^{-x^2\!/2h +ix\xi_0/h 
-t(\xi_0+ix)^\alpha\!/h^\alpha}.
\]
By integrating this asymptotic expansion, we get the following lower bound on the 
left-hand 
side of the observability inequality~\eqref{eq:obs_rfhe}: 
\begin{equation}\label{eq:lower_simple}
 |g_h(T,\cdot)|_{L^2(\set R)}^2 \geq |g_h(T,\cdot)|_{L^2(|x|<\delta)} \geq ch^{-1} 
\int_{|x|<\delta} e^{-x^2\!/h -C h^{-\alpha}}\diff x \geq c' e^{-C' h^{-\alpha}}.
\end{equation}
Again by integrating the asymptotic expansion on $g_h$, we get the following upper on 
the right-hand side of the observability inequality
\begin{equation}\label{eq:upper_simple}
 |g_h|_{L^2([0,T]\times \omega)} \leq |g_h|^2_{L^2([0,T]\times\{|x|>\delta\})}
 \leq Ch^{-1} \int_{[0,T]\times \{|x|>\delta\}} e^{-x^2\!/h -cth^{-\alpha}}\diff 
t\diff x \leq C'e^{-\delta^2\!/h}.
\end{equation}
Comparing this upper bound~\eqref{eq:upper_simple} and the lower 
bound~\eqref{eq:lower_simple}, and taking the limit $h\to 0^+$, we see that the 
observability inequality~\eqref{eq:obs_rfhe} cannot be true if $\alpha<1$. Thus the 
fractional heat equation is not null-controllable.

\section{Some technical computations}\label{sec:technical}
Before we make rigorously the proof outlined in \cref{sec:informal}, we 
carry here some computations as a technical preparation.

\subsection{Perturbation of the saddle point method}\label{sec:saddle_point}

The ``standard'' saddle point method can be stated in the following 
way~\cite[Th.~2.1]{sjostrand_1982}.
\begin{theorem}\label{th:saddle_point_quad}
 Let $U$ be an open bounded neighborhood of $0$ in $\set R^d$. For every 
$N\in\set N$, there exists $C_N>0$ such that for every $h>0$ and every holomorphic 
function $u$ on a complex neighborhood $V$ of $\overline U$, 
 \[
 \int_{U} e^{-x^2\!/2h} u(x) \diff x = \sum_{j=0}^{N-1} 
\frac{h^{d/2+j}}{(2\pi)^{d/2} j!2^j} (\Delta)^j u(0) + R_N(h),
 \]
 where
 \[
 |R_N(h)| \leq C_N \lambda h^{d/2+N} \sup_{z\in V} |u(z)|.
 \]
\end{theorem}
By using the Morse lemma (see for instance~\cite[Lemma C.6.1]{hormander_2007}), one 
can often transform integral of the form $\int e^{\phi(x)/h}u(x)\diff x$ into 
integrals of the form of \cref{th:saddle_point_quad}, plus some exponentially small 
error. Notice that in this theorem, the phase $-x^2\!/2$ does not depend on $h$. 
However, to rigorously treat the integral of Eq.~\eqref{eq:solution_simple}, we need 
to allow the phase to depend on $h$.

\begin{proposition}[Perturbation of the saddle point 
method]\label{th:saddle_point_local}
 Let $h_0>0$ and $\epsilon\colon (0,h_0]\to \set R_+$ such that $\epsilon(h)\to 0$ as 
$h\to 0$. Let $U\subset R$ be an open interval around $0$. Let 
$V$ be a complex open bounded neighborhood of $\overline{U}$ in $\set C$.

For every $0<h\leq h_0$, let $r_h\colon V\to \set C$ such that 
\begin{enumerate}
 \item $\forall 0<h\leq h_0$, $r_h$ is holomorphic on $V$,
 \item there exists $C>0$ such that for every $0<h\leq h_0$ and $\xi\in V$, 
$|r_h(\xi)|\leq C \epsilon(h)$.
\end{enumerate}
For such $r_h$, we define $|r|_\epsilon \coloneqq \inf\{C>0, \forall 0<h\leq h_0, 
\forall \xi\in V, |r_h(\xi)|\leq C\epsilon(h)\}$.

For every $0<h\leq h_0$, let $u_h$ be a holomorphic bounded function on $V$. We 
consider
\[
 I_{h,r}(u) \coloneqq \int_{U} e^{-\xi^2\!/2h +r_h(\xi)/h}u_h(\xi)\diff \xi.
\]

Let $M>0$. We have uniformly in $|r|_\epsilon<M$, and uniformly in $u_h$ holomorphic 
bounded on $V$
\[
 I_{h,r}(u) = \sqrt{2\pi h}\s e^{r_h(0)/h + \bigO(\epsilon(h)^2/h)}\Big(u_h(0) + 
\bigO\big((h+\epsilon(h))|u_h|_{L^\infty(V)}\big)\Big).
\]
\end{proposition}

\begin{proof}[Proof of \cref{th:saddle_point_local}]
The strategy is to see 
$\varphi_{h,r}(\xi) = -\xi^2\!/2 
+r_h(\xi)$ as the phase, and to change the integration 
variable and the integration path to rewrite $I_{h,r}(u)$ in the form 
$I_{h,r}(u) = \int e^{-\xi^2\!/2h} v_h(\xi)\diff \xi$, even if this change of 
variables depends on $h$, and then apply \cref{th:saddle_point_quad}.

In this proof, $M>0$ is fixed. We also choose $V'\subset \set C$ convex such 
that\footnote{We denote $A\csubset B$ for ``$\overline A$ compact and $\overline 
A\subset B$''.} $U\csubset V'\csubset V$. Also, we use the convention that $C$ 
denotes a constant, that depends only on $\epsilon$, $M$, $V$ and $V'$, but not on 
$h$ small enough, $|r|_\epsilon\leq M$ or $\xi\in V'$, and that may be different each 
line.

\step{Step 1: finding the critical point.} We claim that for $h$ small enough, for 
every $|r_h|_\epsilon \leq M$, there exists a unique critical point $\xi_{\cri,h,r}$ 
of $\varphi_{h,r}$ in $\overline{V'}$, and that this critical point is non-degenerate.

Indeed, $\xi\in\overline{V'}$ is a critical point of $\varphi_{h,r}$ if and only if
$\partial_\xi r_h(\xi) = \xi$. But for every\footnote{We will frequently use the 
following standard consequence of the integral Cauchy formula: if $f$ is holomorphic 
on $D(z_0,r)$, then $|f^{(n)}(z_0)|\leq C_{n,r}|f|_{L^\infty(D(z_0,r))}$.} $\xi\in 
V'$, $|\partial_\xi r_h(\xi)|\leq C |r_h|_{L^\infty(V)}\leq C |r|_\epsilon 
\epsilon(h)$. Moreover, since $V'$ is convex, according to the mean value 
inequality, for $\xi,\xi'\in \overline{V'}$, $0<h\leq h_0$ and $|r|_\epsilon \leq M$,
\[
|\partial_\xi r_h(\xi) - \partial_\xi r_h(\xi')| \leq \sup_{V'} |\partial_\xi^2 r_h| 
|\xi-\xi'| \leq C \sup_V |r_h| |\xi-\xi'| \leq C |r|_\epsilon \epsilon(h) |\xi-\xi'| 
\leq CM \epsilon(h) |\xi-\xi'|.
\]
Thus, if $h$ is small enough such that $CM\epsilon(h) <1$, $\xi\mapsto \partial_\xi 
r_h(\xi)$ takes its value in $\overline{V'}$ and is a contraction. Then, according to 
the contraction mapping theorem, there exists a unique $\xi_{\cri,h,r}\in 
\overline{V'}$ such that 
$\partial_\xi r_h(\xi_{\cri,h,r}) = \xi_{\cri,h,r}$. This is the unique critical 
point of $\varphi_{h,r}$ in $\overline{V'}$.

Note that we have $|\xi_{\cri,h,r}|\leq C |r|_\epsilon \epsilon(h)$, where $C$ 
depends only on $\epsilon$, $M$, $V$ and $V'$. Also, the critical value 
$c_{h,r}\coloneqq \varphi_{h,r}(\xi_{\cri,h,r})$ satisfies 
\[|c_{h,r}-r_h(0)| = \Big|\frac{-\xi_{\cri,h,r}^2}2 + r_h(\xi_{\cri,h,r}) 
-r_h(0)\Big| \leq \frac{|\xi_{\cri,h,r}|^2}4 + 
|\xi_{\cri,h,r}|\sup_{V'}|\partial_\xi r_h| \leq C \epsilon(h)^2 |r|_\epsilon.
\]

Moreover, we have $|\partial_\xi^2 \varphi_{h,r}(\xi_{\cri,h,r}) +I| = 
|\partial_\xi^2 r_h(\xi_{\cri,h,r})| \leq C|r|_\epsilon \epsilon(h)$. So, if $h$ is 
small enough, the critical point $\xi_{\cri,h,r}$ is nondegenerate.

\step{Step 2: change of variables.}
Now that we know where the critical point is, and what the critical value of 
the phase is, we want to change the intergation variables to rewrite $I_{h,r}(u)$ 
in the form $I_{h,r}(u) = \int e^{-\xi^2\!/2h}\tilde u_h(\xi)\diff \xi$.

We define $\psi_{h,r}(\xi)$ on $V'$, for $h$ small enough by 
\[\psi_{h,r}(\xi+\xi_{\cri,h,r}) = 
\xi\left(2\int_0^{1}\partial_x^2 \varphi_{h,r}(s\xi)\diff 
s\right)^{1/2}.
\]
According to Taylor's formula, for every $\xi\in V'$, we 
have $\varphi_{h,r}(\xi) =c_{h,r}+ \psi_{h,r}(\xi)^2\!/2$.

So, by the change of variables/integration path $\eta =\psi_{h,r}(\xi)$, we 
have:
\begin{equation}
 I_{h,r}(u) 
 = e^{c_{h,r}/h}\int_{\psi_{h,r}(U)}e^{-\eta^2\!/{2h}} u_h(\xi(\eta)) \frac{\diff[0] 
\xi}{\diff[0] \eta} \diff \eta =
 e^{c_{h,r}/h}\int_{\psi_{h,r}(U)}e^{-\eta^2\!/{2h}} \tilde u_h(\xi(\eta))\diff \eta,
\end{equation}
where $\tilde u_h(\eta) \coloneqq u_h(\xi(\eta))\diff\xi/\diff[-6]\eta$.

Note that according to the definition of $\psi_{h,r}$, for $\xi \in V'$, we have 
$|\psi_{h,r}(\xi) - \xi|\leq C\epsilon(h) 
|r|_\epsilon$. Thus, we have for every $\eta\in \psi_{h,r}(V')$, 
$|\psi_{h,r}^{-1}(\eta) - \eta|\leq C \epsilon(h)|r|_\epsilon$. Thus, $\diff 
\eta/\diff[-6] \xi = 1 +\bigO(\epsilon(h)|r|_\epsilon)$.

\step{Step 3: conclusion.}
So, by the standard saddle point method (\cref{th:saddle_point_quad}):
\begin{equation}\label{eq:saddle_precise}
\begin{aligned}
 I_{h,r}(u) &= e^{c_{h,r}/h}\sqrt{2\pi h}\left(\tilde u_h(0)+\bigO\big(h|\tilde 
u_h|_{L^\infty(V)}\big)\right)\\
&= e^{c_{h,r}/h}\sqrt{2\pi h}\left(u_h(0)+\bigO\big((h+\epsilon(h))| 
u_h|_{L^\infty(V)}\big)\right)
\end{aligned}
\end{equation}
and since $c_{h,r}/h = r_h(0)/h + \bigO(\epsilon(h)^2/h)$ the Proposition is 
proved.
\end{proof}

\subsection{The framework: truncated coherent states}
As we said in the introduction, our aim is to prove the lack of null-controllability 
of several equations, that behave in some sense like the fractional heat equation. 
However, treating these equations requires more precise estimates than the fractional 
heat equation does.

To avoid making similar computations several times, we do them in a somewhat 
general framework. This way, we will be able to treat the other equations 
(Kolmogorov, 
etc.) directly.

\begin{hypothesis}\label{hyp}
We consider the following domain, constants and functions:
\begin{enumerate}
 \item let $K>0$ and $\Cc = \{\xi \in \set C, \Re(\xi)>K, 
|\Im(\xi)|<K^{-1}\Re(\xi)\}$,
 \item let $\xi_0>0$ large enough (for instance $\xi_0 = 4(K+1)$),
 \item let $\delta>0$ small enough such that for every $\xi\in\set R$ and $x\in 
\set R$, $|\xi-\xi_0|<\delta$ and $|x|<\delta$ implies $\xi+\xi_0 + ix \in \Cc$,
 \item let $\chi\in C_c^\infty(-\delta,\delta)$ such that $0\leq \chi\leq 1$ and 
$\chi\equiv1$ on a neighborhood of $0$, say $(-\delta_2,\delta_2)$,
 \item let $X$ be a topological space, and $0<h_0<1$, and for every $\gamma\in X$ and 
$0<h\leq h_0$, let $\rho_{\gamma,h}\colon \Cc\to \set C$ be a holomorphic function,
 \item we assume that uniformly in $\gamma \in X$ and $0<h\leq 
h_0$, $\rho_{\gamma,h}(\xi) = \smallo(|\xi|)$ in the limit $|\xi|\to +\infty$, 
$\xi\in \Cc$,
 \item finally, for every $0<h\leq h_0$, we define
  \[
   \epsilon(h) \coloneqq \sup_{|\xi|<\delta, |x|<\delta, \gamma\in X} h 
\left|\rho_{\gamma,h}\left(\frac{\xi+\xi_0+ix}h\right)\right|.
  \]
\end{enumerate}
\end{hypothesis}

The goal of the next subsections is to get upper and lower bounds on the following 
integral, where $(v_h)$ is a family of bounded holomorphic functions on $\Cc$:
\begin{equation}\label{eq:defI}
 \Ighv = \int_{\set R} \chi(\xi-\xi_0) e^{-(\xi-\xi_0)^2\!/2h +ix\xi/h + 
\rho_{\gamma,h}(\xi/h)} v_h(\xi)\diff \xi.
\end{equation}

\begin{remark}
 \begin{enumerate}
  \item The hypothesis 6 implies that $\epsilon(h)\to 0$ as $h\to 0$.
  \item For instance, if $\rho$ is independent of $\gamma$ and $h$, with $\rho(\xi) = 
|\xi|^{\alpha}$ and $0\leq \alpha < 1$, then we have $C^{-1}h^{1-\alpha} < 
\epsilon(h) < C h^{1-\alpha}$ for some $C>0$.
  \item In the applications, we will typically choose $X = [0,T]$ and for $t\in X$, 
$\rho_{t,h}(\xi) = -t \rho(\xi)$ with some $\rho \colon \Cc \to \set C$ such that 
$\rho(\xi) = \smallo(|\xi|)$. We will also usually choose $v_h = 1$. In that 
case, $g_h(t,x) \coloneqq I_{t,h,1}(x)$ is solution of $(\partial_t + 
\rho(\sqrt{-\Delta}))g_h = 1$ with initial condition $g_{0,h}(x) = \sqrt{2\pi 
h}\s\chi(-ih\partial_x -\xi_0) e^{-x^2\!/2h+ix\xi_0/h}$, which belongs in $L^2(\set 
R)$. However, some applications will require a larger parameter space $X$ and $v_h 
\neq 1$.
 \end{enumerate}
\end{remark}

\subsection{Asymptotics for the evolution of coherent states}
\begin{proposition}\label{th:ighv_asy}
 Assuming \cref{hyp}, we have uniformly in $\gamma\in X$, $|x|$ 
small enough, $\gamma\in X$  and $v_h\colon \Cc \to \set C$ holomorphic bounded
\begin{align*}
 \Ighv &= \int_{\set R} \chi(\xi-\xi_0) e^{-(\xi-\xi_0)^2\!/2h +ix\xi/h + 
\rho_{\gamma,h}(\xi/h)} v_h(\xi)\diff \xi \\
&= \sqrt{2\pi h}\s e^{ix\xi_0/h - x^2\!/2h +\rho_{\gamma,h}\big(\frac{\xi_0+ix}h\big) 
+ \bigO\big(\frac{\scriptstyle\epsilon(h)^2}h\big)} 
\big(v_h\left(\xi_0+ix\right) + 
\bigO\big((h+\epsilon(h))|v_h|_{L^\infty(\Cc)}\big)\Big),
\end{align*}
in the limit $h\to 0^+$.
\end{proposition}
\begin{remark}\label{rk:ighv_asy}
 For most of the applications, we don't care about the term $\rho((\xi_0+ix)/h)$, 
apart from the fact it is $\bigO(\epsilon(h)/h)$. Moreover, we usually have $v_h(\xi) 
= 1$, or at least $v_h(\xi) \to 1$ as $h\to 0$, uniformly in $\xi\in \Cc$. Under 
this condition, \cref{th:saddle_point_local} implies the slightly less precise 
asymptotic expansion
\[
 \Ighv 
= \sqrt{2\pi h}\s e^{ix\xi_0/h - x^2\!/2h + 
\bigO\big(\frac{\scriptstyle\epsilon(h)}h\big)} 
\left(1 + \smallo(1)\right),
\]
which will be enough in most cases.
\end{remark}

\begin{proof}
 The idea is that this integral has almost the form of \cref{th:saddle_point_local}. 
Let us actually rewrite it as such.
 
 \step{Step 1: change of integration path.}
 Note that $\chi(\xi-\xi_0) \equiv 1$ for $\xi_0-\delta_2 < \xi < \xi_0+2\delta_2$, 
so $\chi(\xi-\xi_0)$ extends holomorphically to $|\Re(\xi)-\xi_0|<\delta_2$ by $1$. 
Moreover, if $\xi \in \set C$ and $|\Re(\xi)-\xi_0|<\delta_2$, $|\Im(\xi)|<\delta_2$, 
then, according to \cref{hyp} item 3, $\xi\in \Cc$.
 We deduce that the integrand of $\Ighv$ is holomorphic on $\{\xi\in \set C, 
|\Re(\xi)-\xi_0|<\delta_2, |\Im(\xi)|<\delta_2\}$.

Thus, we can change the integration path of $\Ighv$, as long as we modify it only 
between $\xi_0-\delta_2$ and $\xi_0+\delta_2$, and that the modified part of the 
integration path stays inside $\{|\Re(\xi)-\xi_0|<\delta_2, |\Im(\xi)|< \delta_2\}$.
\begin{figure}
 \begin{center}
  \begin{tikzpicture}[baseline = (O), scale=0.6]
 \node (O) at (0,0){};
 \draw (0,0.7) node(C){} \cross{0.2};
 \draw (C) +(45:4) -- +(225:2)
       (C) +(-45:2) -- +(135:4);
 \node at(2.5,0|-C) {$\Re(\phi_x)>0$};
 \node at(-2.5,0|-C) {$\Re(\phi_x)>0$};
 \node at(-0,-1) {$\Re(\phi_x)<0$};
 \draw [->] (-4,0) -- (4,0);
 \draw [->] (0,-0.5) -- (0,4);
 \draw[|-|, very thick, blue] (-2,0)node(C1){} -- (2,0) node(C2){};
 \draw[->, blue] (2,-0.5) node[below right, blue]{$\chi^{-1}(1)$} -- (1.2,0);
 \draw[red, thick, postaction={decorate}, decoration={markings, mark = between 
positions 0.2 and 0.8 step 0.2 with{\arrow{>}}}] (-4,0) -- 
	(C1.center) .. controls (-0.7,0) and (-0.9,0.7) .. 
	($(C.center)+(-0.4,0)$) -- ($(C.center)+(0.4,0)$)
	.. controls (0.9,0.7) and (0.7,0) .. 
	(C2.center) -- (4,0);
 \node[right=0.1] at (C) {$ix$};
\end{tikzpicture}%
 \end{center}
 \caption{In blue, the interval where $\chi = 1$. The diagonal lines define four 
sectors; in the left and right ones, $\Re(\varphi_x) > 0$ and in the top and bottom 
ones, $\Re(\phi_x)<0$. In red, the path of integration we chose in the integral 
defining $\Ighv$ (Eq.~\ref{eq:defI}). If $|x|$ is small enough, we choose a path that 
goes through the saddle point $ix$, but that stays in 
$\{\Re(\varphi_x)>0\}$.}\label{fig:saddle_point}
\end{figure}

Let $\xi_\cri = \xi_0 + ix$ be the critical point of $\varphi_x(\xi) \coloneqq 
-(\xi-\xi_0)^2\!/2 +ix\xi$. We choose an integration path $\Gamma$ parametrized by 
$\Gamma(t) = t+ \xi_0 +ix \chi_2(t)$, where $\chi_2\in 
C_c^\infty(-\delta_2,\delta_2)$ 
with $0\leq \chi_2\leq 1$ and $\chi_2\equiv 1$ on a neighborhood of $0$ (say 
$(-\delta_3,\delta_3)$). Then, we have
\begin{align*}
 \Ighv &= \int_{\Gamma} \chi(\Re(z)) e^{\varphi_x(z)/h - \rho_{\gamma,h}(z/h)} v_h(z) 
\diff z\\
 &= \int_{-\delta_3}^{\delta_3} e^{-t^2\!/2h - x^2\!/2h + ix\xi_0/h 
-\rho_{\gamma,h}\left(\frac{\xi_0 + ix +t}h\right)} 
v_h\left(\xi_0+ix+t\right) \diff t + R_{\gamma,h,v}(x),
\end{align*}
where we used that for $\delta_3<t<\delta_3$, $\varphi_x(\Gamma(t)) = -(t+ix)^2\!/2 
+ix(t+\xi_0+ix) = -t^2\!/2-x^2\!/2+ix\xi_0$, and where $R_{\gamma,h,v}(x)$ is the 
part of the integral out of $(-\delta_3,\delta_3)$.

\step{Step 2: upper-bound for the remainder.} Since $\supp(\chi)\subset 
(-\delta,\delta)$, so that the integrand is $0$ for $|t|>\delta$, the remainder 
$R_{\gamma,h,v}(x)$ is upper-bounded by
\[
 |R_{\gamma,h,v}(x)| \leq 2\delta e^{-\delta_3^2\!/2h + 
\epsilon(h)/h}|\chi_2'|_{L^\infty}.
\]
where we used the definition of $\epsilon$ (\cref{hyp} item 7) to bound 
$\rho_{\gamma,h}$. 
Moreover, $\epsilon(h)\to 0$, so we have (for instance)
\begin{equation}\label{eq:upper_remainder}
 |R_{\gamma,h,v}(x)| \leq C e^{-\delta_3^2\!/4h}|v_h|_{L^\infty}.
\end{equation}

\step{Step 3: asymptotic expansion for the integral in $(-\delta_3,\delta_3)$.}
To get an asymptotic expansion of the part of the integral between $-\delta_3$ and 
$\delta_3$, we can apply \cref{th:saddle_point_local}. Indeed, for $0<h\leq h_0$, 
$\gamma\in X$, $|x|<\delta_3/2$ and $\xi$ in a small enough complex neighborhood $U$ 
of $[-\delta_3,\delta_3]$, let
\[
 r_{\gamma,x,h}(\xi) = h\rho_{\gamma,h}\left(\frac{\xi_0+ix+\xi}h\right).
\]
By definition of $\epsilon(h)$, we have $|r_{\gamma,h,x}|<\epsilon(h)$, or, in other 
words, $|r_{\gamma,x,h}|_\epsilon \leq 1$. For the same parameters, we also define
\[
 u_{x,h}(\xi) = v_h\left(\xi_0+ix+\xi\right).
\]
Then, according to \cref{th:saddle_point_local},
\begin{align*}
 \Ighv &= \int_{-\delta_3}^{\delta_3} e^{-x^2\!/2h +ix\xi_0/h -t^2\!/2h + 
r_{\gamma,x,h}(t)/h} 
u_{x,h}(t) \diff t + R_{\gamma,h,v}(x) \\
 &= \sqrt{2\pi h}\s e^{-x^2\!/2h + ix \xi_0/h + r_{\gamma,x,h}(0)/h}\Big(u_{x,h}(0)+ 
\bigO((h+\epsilon(h))|u_{x,h}|_{L^\infty(U)}\big)\Big)+ R_{\gamma,h,v}(x)
\end{align*}
uniformly in $\gamma\in X$, $|x|<\delta_3/2$ and $v_h$ holomorphic bounded on $\Cc$.

\step{Step 4: conclusion.}
With the upper-bound~\eqref{eq:upper_remainder} on $R_{\gamma,h,v}(x)$, the claimed 
asymptotic expansion follows.
\end{proof}

\subsection{Upper bounds for the evolution of coherent states}
We will also need upper bounds for $\Ighv$ that hold for large $x$.
\begin{proposition}\label{th:ighv_upper}
 We assume~\cref{hyp}, except that the item 6 is only assumed to hold 
\emph{locally} uniformly in $\gamma\in X$ (and uniformly in $0<h\leq h_0$). We define 
for $\gamma\in X$ and $0<h\leq h_0$
\[
 \epsilon_\gamma(x) \coloneqq \sup_{|\xi|<\delta,|x|<\delta} h\Re 
\left(\rho_{\gamma,h}\left(\frac{\xi+\xi_0+ix}h\right)\right).
\]
Let $\eta>0$. For every $N>0$, there exist $c,C>0$ such that for every $|x|>\eta$, 
$\gamma\in X$ and $v_h$ holomorphic bounded on $\Cc$,
\[
 |\Ighv| \leq \frac C{|x|^N} 
e^{-c/h+\epsilon_\gamma(h)/h}|v_h|_{L^\infty(\Cc)}.
\]
\end{proposition}
\begin{remark}\label{rk:ighv_upper}
 \begin{enumerate}
  \item 
In the applications, we will typically choose $X = \set R_+$ and 
$\rho_{t,h}(\xi) = -t \rho(\xi)$.
  \item For instance, consider the case $X = \set R_+$ and $\rho_{t,h}(\xi) = -t z 
\xi^\alpha$, where $\Re(z)>0$. This choice is relevant to the equation 
$(\partial_t + z (-\Delta)^{\alpha/2})f(t,x) = \mathds 1_\omega u(t,x)$. If 
$K>0$ is large enough, for every $\xi\in \Cc$, $\Re(z\rho(\xi))\geq c 
|\xi|^\alpha$ for some $c>0$.\footnote{Indeed, if $z = r_0 e^{i\theta_0}$, then 
$|\theta_0|<\pi/2$. And if $\xi = re^{i\theta}$, then $\Re(z\xi^\alpha) = 
r_0r^{\alpha}\cos(\alpha\theta+\theta_0)$. But if $\xi \in \Cc$, then 
$r\lvert\sin(\theta)\rvert \leq r K^{-1} \cos(\theta)$, or $|\theta|\leq 
\atan(K^{-1})$. So, if $K$ is large enough, for every $\xi = re^{i\theta}\in \Cc$, 
$|\alpha \theta + \theta_0| \leq \pi/2-\tau$ for some $\tau>0$. Then, 
$\Re(z\xi^\alpha) = r_0 r^\alpha \cos(\alpha\theta+\theta_0) \geq r_0 r^\alpha 
\cos(\pi/2-\tau) = c|\xi|^\alpha$.} Then $\epsilon_t(h) \leq -ct 
h^{1-\alpha}$. \end{enumerate}
\end{remark}

\begin{proof}
\step{Step 1: integration by parts.}
 First, we integrate by parts to get the decay in $x$. As in the previous proof, we 
denote $\varphi_x(\xi) = -(\xi-\xi_0)^2\!/2 +ix\xi$. We remark that 
$h\partial_\xi e^{\varphi_x(\xi)/h} = -(\xi-\xi_0-ix)e^{\varphi_x(\xi)/h}$. Thus, 
with $L_x = \frac1h\partial_\xi \frac1{\xi-\xi_0-ix}$, we have
\[
 \Ighv = \int_{-\infty}^{+\infty} e^{-\varphi_x(\xi)} L_x^N\left(\chi(\xi-\xi_0) 
e^{\rho_{\gamma,h}\big(\frac\xi h\big)}v_h\left(\xi \right)\right)\diff \xi.
\]

\step{Step 2: change of integration path.}
Next, as in the proof of \cref{th:ighv_asy}, we can change the integration path 
between $\xi_0-\delta_2$ and $\xi_0+\delta_2$, as long as this modification stays 
inside $\{|\Re(\xi)-\xi_0|<\delta_2, |\Im(\xi)|<\delta_2\}$. We choose $\chi_2$ as in 
the proof of \cref{th:ighv_asy}, i.e.\ $\chi_2\in C_c^\infty(-\delta_2,\delta_2)$, 
$0\leq \chi_2\leq 1$, $\chi_2 \equiv 1$ on $(-\delta_3,\delta_3)$. Then we choose the 
path $\Gamma(t) = t+i\eta_2 \sign(x)\chi_2(t-\xi_0)$, where $\eta_2>0$ is small enough, 
for instance $\eta_2 = \min(\eta/2,\delta_2/2)$ (see \cref{fig:saddle_point1}).
\begin{figure}
 \begin{center}
  \begin{tikzpicture}[baseline = (O), scale = 0.6]
 \def\yzero{3}
 \def\yun{0.8}
 \def\xzero{1.5}
 \def\b{2.598}
 \def\tzero{0.5493}
 \node (O) at (0,0){};
 \draw (0,\yzero) node(C){} \cross{0.2};
 \node (C0) at (0,\yun){};
 \node[right=0.1] at (C) {$ix$};
 \draw (C) +(45:2) -- +(225:5)
       (C) +(-45:5) -- +(135:2);
 \draw [->] (-4,0) -- (4,0);
 \draw [->] (0,-1) -- (0,4);
 \node (C1) at (-\xzero,0) {};
 \node (C2) at (\xzero,0) {};
 \draw[|-|, very thick, blue] (-2,0) -- (2,0);
 \draw[->, blue] (1,-0.5) node[below right, blue]{$\chi^{-1}(1)$} -- (0.5,0);
 \draw[red, thick, postaction={decorate}, decoration={markings, mark = between 
positions 0.2 and 0.8 step 0.2 with{\arrow{>}}}] (-4,0) -- 
	(C1.center) .. controls (-0.7,0) and (-0.9,0.7) .. 
	($(C0.center)+(-0.4,0)$) -- ($(C0.center)+(0.4,0)$)
	.. controls (0.9,0.7) and (0.7,0) .. 
	(C2.center) -- (4,0);
 \draw[red,->] (-3,-0.5) node[below]{$\Gamma$}-- (-2.6,0);
\end{tikzpicture}%
 \end{center}
 \caption{As in \cref{fig:saddle_point}, the interval where $\chi = 1$ in blue. If 
$x$ is not too small, we deform a bit the integration path toward $ix$. 
}\label{fig:saddle_point1}
\end{figure}

\step{Step 3: upper-bound for $e^{\varphi_x(\xi)}$.}
On this path $\Gamma$, for every $|x|>\eta$ and $0<h\leq h_0$,
\begin{align*}
 |e^{\varphi_x(\Gamma(t))/h}| &= e^{\Re(\varphi_x(\Gamma(t)))} \\
  &= e^{(-t^2\!/2 +\eta_2^2\chi_2^2(t-\xi_0)/2 -|x|\eta_2\chi_2(t-\xi_0))/h}\\
  &\leq e^{-t^2\!/2h -|x|\eta_2\chi_2(t-\xi_0)/2h},
\end{align*}
where we used that $|x|>\eta_2$ and $0\leq \chi_2\leq 1$. Thus, for some $c>0$, we 
have for every $|x|>\eta_2$ and $t\in \set R$
\begin{equation}\label{eq:ighv_upper_phase}
|e^{\varphi_x(\Gamma(t))/h}| \leq e^{-c/h}.
\end{equation}

\step{Step 4: upper bound for the rest of the integrand.} We claim that there exists 
$C_N>0$ such that for every $f$ $C^\infty$ on $\Gamma$, for every $|x|>\eta$ and 
$\xi\in \Gamma$, 
\begin{equation}\label{eq:upper_Lx}
 |L_x^N f(\xi)| \leq \frac{C_N}{|hx|^N} \sum_{k\leq N}|\partial_\xi^k f(\xi)|.
\end{equation}

Indeed, according to Leibniz' rule, for any $k\geq 0$, $\xi\in \Gamma$ and $f$ 
$C^\infty$ on $\Gamma$,
\[
 \partial_\xi^k L_xf(\xi) = h^{-1} \partial_\xi^{k+1} \frac{f(\xi)}{\xi-\xi_0-ix} 
 =h^{-1}\sum_{\ell \leq k+1} C_{k,\ell} \frac{\partial_\xi^\ell 
f(\xi)}{(\xi-\xi_0-ix)^{k+2-\ell}}.
\]
So, reminding that $|x|>\eta$ and $|\Im(\xi)|<\frac\eta2$
\[
 |\partial_\xi^k L_x f(\xi)| \leq \frac{C_k}{|hx|} \sum_{\ell \leq k+1} 
|\partial_\xi^\ell 
f(\xi)|.
\]
By iterating this estimate, we get the upper bound~\eqref{eq:upper_Lx}. Now, choosing 
$f(\xi) = \chi(\xi- \xi_0)e^{-\rho_{\gamma,h}(\xi/h)}v_h(\xi)$, we get
\begin{align*}
 |L_x^N(\chi(\xi-\xi_0)e^{\rho_{\gamma,h}(\xi/h)} v_h(\xi))|
 &\leq \frac{C_N}{|hx|^N}\sum_{k\leq N} \left|\partial_\xi^k 
\left(\chi(\xi-\xi_0)e^{\rho_{\gamma,h}(\xi/h)}v_h(\xi)\right)\right|\notag\\
 &\leq \frac{C'_N}{|hx|^N}\sum_{k\leq N} \left|\partial_\xi^k 
\left(e^{\rho_{\gamma,h}(\xi/h)}v_h(\xi)\right)\right|.
\end{align*}
Moreover, for any $f$ holomorphic on $\Cc$, and for any $\xi\in \Gamma$ and $r>0$ 
such that $D(\xi,r)\subset \Cc$, the Cauchy integral formula implies that 
$|\partial_\xi^k f(\xi)|\leq C_r |f|_{L^\infty(D(\xi,r))}$. So,
\begin{align}
 |L_x^N(\chi(\xi-\xi_0)e^{\rho_{\gamma,h}(\xi/h)} v_h(\xi))|
 &\leq  \frac{C''_N}{|hx|^N} \left| 
e^{\rho_{\gamma,h}(\xi/h)}v_h(\xi)\right|_{L^\infty(D(\xi,r))}\notag\\
 &\leq  \frac{C''_N}{|hx|^N}  
e^{\epsilon_\gamma(h)/h}|v_h|_{L^\infty(\Cc)}.\label{eq:upper_symbol}
\end{align}

\step{Step 5: conclusion.}
Putting together the bounds~\eqref{eq:upper_symbol} and~\eqref{eq:ighv_upper_phase}, 
we get
\[
 |\Ighv|\leq \frac{C_N}{|x|^N} h^{-N}e^{-c/h + 
\epsilon_\gamma(h)/h}|v_h|_{L^\infty(\Cc)},
\]
which implies the claimed estimate.
\end{proof}

We also have the following upper-bound, which is weaker but valid for any $x$, even 
small. We will need it for some applications.
\begin{proposition}\label{th:ighv_upper_simple}
 Under the same hypotheses as \cref{th:ighv_upper}, there exist $c,C>0$ such that 
for every $x\in\set R$, $\gamma\in X$ and $v_h$ holomorphic bounded on $\Cc$,
\[
 |\Ighv| \leq Ce^{\epsilon_\gamma(h)/h}|v_h|_{L^\infty(\Cc)}.
\]
\end{proposition}
\begin{proof}
 It is only the integral triangle inequality.
\end{proof}

\section{Non-null-controllability of the generalized fractional heat 
equation}\label{sec:gen_frac}

\subsection{The generalized fractional heat equation on the whole real 
line}\label{sec:heat_R}

\begin{proof}[Proof of \cref{th:gen_frac} in the case $\Omega = \set R$]
\step{Well-posedness.}
Let us recall that $\inf_{\set R_+} \Re(\rho) < +\infty$. So, denoting $M$ this 
infinimum, we have for every $t\geq 0$, $\sup_{\xi\in \set R} |e^{-t\rho(|\xi|)}|\leq 
e^{tM}$. Thus, for every $t\geq 0$, we can define a linear bounded operator on 
$L^2(\set R)$ by
\[
 \forall f_0\in L^2(\set R),\ \forall x\in \set R,\ S(t)f_0(x) = \mathcal 
F^{-1}(e^{-t\rho(|\xi|)}\mathcal 
Ff_0)(x) = \frac1{2\pi}\int_{\set R} e^{ix\xi-t\rho(|\xi|)} \mathcal Ff_0(\xi) \diff 
\xi.
\]
We can see that $S(t)$ is a strongly continuous semigroup of bounded operators on 
$L^2(\set R)$. Moreover, the infinitesimal generator of $S(t)$ is 
$\rho(\ssqrt{-\Delta})$. Thus, the equation~\eqref{eq:gen_frac_control} is 
well-posed, in the sense of semigroups (see for instance 
\cite[Def.~2.36 and Th.~2.37]{coron_2007}).

\step{Construction of the counterexample to the observability inequality.}
We remind that the null-controllability of the generalized fractional heat 
equation on $\omega$ and in time $T$ is equivalent to the following 
\emph{observability inequality}~\cite[Th.~2.44]{coron_2007}: there exists $C>0$ such 
that for every $g_0\in L^2(\set R)$, the solution $g$ of
\begin{equation}
 \label{eq:gen_frac_adj}
 \partial_t g + \overline\rho(\sqrt{-\Delta})g = 0,\quad g(0,\cdot) = g_0
\end{equation}
satisfies
\begin{equation}
 \label{eq:gen_frac_obs}
 |g(T,\cdot)|_{L^2(\set R)} \leq C |g|_{L^2([0,T]\times \omega)}.
\end{equation}

In \cref{hyp} item 1, we choose $K$ and $\Cc$ to be those of the statement 
of~\cref{th:gen_frac}. Let $\xi_0, \delta$ and $\chi$ as in \cref{hyp} item 
2--3. Then, for $h>0$, we consider $g_{0,h}\in L^2(\set R)$ defined by
\[
 g_{0,h}(x) =  \sqrt{2\pi h}\s\chi(-ih\partial_x -\xi_0)e^{-x^2\!/2h +ix\xi_0/h} = 
h\int_{\set R} \chi(h\xi-\xi_0)e^{-(h\xi-\xi_0)^2\!/2h +ix\xi}\diff \xi.
\]

The solution $g_h$ of the generalized fractional heat 
equation~\eqref{eq:gen_frac_adj} with this initial condition is
\begin{equation}\label{eq:gen_frac_coherent}\begin{aligned}
 g_h(t,x) &= h\int_{\set R} \chi(h\xi-\xi_0) e^{-(h\xi-\xi_0)^2\!/2h +ix\xi 
-t\overline \rho(\xi)} \diff \xi\\
&= \int_{\set R} \chi(\xi-\xi_0) 
e^{-(\xi-\xi_0)^2\!/2h +ix\xi_0/h - t\overline \rho(\xi/h)}\diff \xi
\end{aligned}
\end{equation}
(let us remind that according to \cref{hyp}, $\chi(h\xi-\xi_0)$ is zero for 
$|h\xi-\xi_0|\geq\delta$, and in particular for $\xi<0$ if $\delta$ is chosen small 
enough, as in \cref{hyp}).

\step{Conclusion.}
In \cref{hyp} item 5, we choose $X = [0,T]$. For $t\in X$ and $h>0$, we choose 
$\rho_{t,h} \colon \xi\in\Cc \mapsto -t\overline \rho(\overline \xi)$. Since $\rho$ 
is holomorphic on $\Cc$ with $\rho(\xi) = \smallo(|\xi|)$, so is $\rho_{t,h}$ for 
every $t\in X$ and $h>0$. In other words \cref{hyp} item 5--6 are satisfied. 
Moreover, with the notations of Eq.~\eqref{eq:defI}, the function $g_h$ given 
by~\eqref{eq:gen_frac_coherent} can be writen as $g_h(t,x) = I_{t,h,1}(x)$.

So, according to \cref{th:ighv_asy} (or more precisely \cref{rk:ighv_asy}), there 
exists $C,c>0$ such that for $t\in[0,T]$, $x$ small enough (say $|x|<\eta'$) and $h>0$ 
small enough
\begin{equation}
 \label{eq:lower_coherent}
 |g_h(t,x)| \geq \frac12 e^{-x^2\!/2h -C\epsilon(h)/h}.
\end{equation}
Moreover, according to \cref{th:ighv_upper}, there exists $C,c>0$ 
such that for $t\geq 0$, $|x|>\eta$ and $h>0$ small enough,
\begin{equation}
 \label{eq:upper_coherent}
 g_h(t,x) \leq \frac C{|x|^2}e^{-c/h}.
\end{equation}

Thus, we have according to the lower bound~\eqref{eq:lower_coherent}
\begin{align}
 |g_h(T,\cdot)|_{L^2(\set R)} &\geq |g_h(T,\cdot)|_{L^2(|x|<\eta')} \geq 
c h^{1/4} e^{-C\epsilon(h)/h}\\
\intertext{and according to the upper bound~\eqref{eq:upper_coherent},}
 |g_h|_{L^2((0,T)\times \omega)}^2 &\leq T 
\int_{|x|>\eta}\frac{C}{x^4}e^{-c/h}\diff x \leq 
Ce^{-c/h}.
\end{align}
and since $\epsilon(h)\to 0$ as $h\to 0$, taking $h\to 0$ disproves the 
observability inequality.
\end{proof}

\begin{remark}
We implicitly looked at the generalized fractional heat equation with complex 
valued solutions. This means that we proved that 
there exists an initial condition $f_0$ of the generalized fractional heat equation 
that we cannot steer to $0$, but this initial condition might not be real valued. In 
the case where $\rho(\set R_+)\subset \set R_+$, we might be more interested 
in real valued solutions. But our results actually implies there exists a real valued 
initial condition that cannot be steered to $0$, for if both the real part $\Re(f_0)$ 
and the imaginary part $\Im(f_0)$ could be steered to $0$, then $f_0$ itself could be 
steered to $0$. A similar arguments stays valid for the Kolmogorov-type equation.
\end{remark}

\subsection{The generalized fractional heat equation on the 
torus}\label{sec:heat_bounded}
The case of the generalized fractional heat equation on the torus is a bit different 
because we are not dealing with integrals, but sums. Therefore, tools like the saddle 
point method do not seem to be of much use. Nonetheless, with a trick, we can deduce 
the theorem on the torus from the theorem on the whole real line.

\begin{proof}[Proof of \cref{th:gen_frac} in the case $\Omega = \set T$]
The basic idea is the trick of the proof of Poisson summation formula, namely the 
fact that the Fourier coefficients of a function of the form $\per{g_0}(x) = 
\sum_{k\in \set Z} g_0(x+2\pi k)$ are the values of the Fourier transform of $g_0$ 
evaluated at the integers (up to a multiplication by $\sqrt{2 \pi}$).

So, let $g_h\in C^\infty (\set R)$ be as in the previous section. Since the Fourier 
transform of $g_h(t,\cdot)$ is $C^\infty$ with compact support,\puncfootnote{We 
added the cutoff function $\chi$ just to localize the Fourier transform away from the
singularity of $|\xi|^\alpha$ at $\xi = 0$.} $g_h(t,x)$ decays faster than any 
polynomials as $|x|\to \infty$ and we can define $\per{g_h}(t,x) 
= \sum_{k\in \set Z}g_h(t,x+2\pi k)$. According to the trick described before, 
$c_n(\per{g_h}(t,\cdot)) = (2\pi)^{-1/2} \mathcal F(g_h)(t,\cdot)(n)$. But, by 
definition of $g_h$ as the solution of the rotated fraction heat equation, 
$\mathcal F(g_h)(t,\cdot)(\xi) = \mathcal F(g_h)(0,\cdot)(\xi) e^{-t\overline 
\rho(|\xi|)}$, so, using the trick again:
\begin{equation}
 c_n(\per{g_h}(t,\cdot)) = c_n(\per{g_h}(0,\cdot))e^{-t\overline{\rho}(|n|)}.
\end{equation}
So $\per{g_h}$ is a solution to the generalized fractional heat 
equation~\eqref{eq:gen_frac_adj} on the torus. Now we prove that 
the terms for $k\neq 0$ are negligible. Indeed, we have by definition of 
$\per{g_h}$
\begin{align}
  |\per{g_h}(T,\cdot)|_{L^2(\set T)} &= \Big\lvert\sum_{k\in\set Z} g_h(T,\cdot + 
2\pi k)\Big\rvert_{L^2(\set T)}\\
\intertext{and by singling out to term for $k=0$ and thanks to the triangle 
inequality}
  |\per{g_h}(T,\cdot)|_{L^2(\set T)} &\geq |g_h(T,\cdot)|_{L^2(-\pi,\pi)} -
   \sum_{k\neq 0}|g_h(T,\cdot) |_{L^2((2k-1)\pi,(2k+1)\pi)}\\
\intertext{and thanks to the pointwise estimates on $g_h$ 
(Eq.~\eqref{eq:upper_coherent} and~\eqref{eq:lower_coherent})}
  |\per{g_h}(T,\cdot)|_{L^2(\set T)} 
  &\geq ch^{1/4}e^{-C\epsilon(h)/h} - \sum_{k\neq 0} \frac{C}{k^2}e^{-c/h}
  \geq ch^{1/4}e^{-C\epsilon(h)/h} - Ce^{-c/h}.
\end{align}

In the same spirit, we have thanks to the triangle inequality, and identifying 
$\omega = \set T\setminus [-\epsilon,\epsilon]$ with
$(-\pi,\pi)\setminus [-\epsilon,\epsilon]\subset \set R$
\begin{align}
 |\per{g_h}|_{L^2([0,T]\times \omega)}
 &\leq\sum_{k\in \set Z}|g_h|_{L^2([0,T]\times(\omega+2\pi k))}\\
 \intertext{and according to the estimate~\eqref{eq:upper_coherent},}
 |\per{g_h}|_{L^2([0,T]\times \omega)}
 &=\mathcal O(e^{-c/h}).
\end{align}

Taking $h\to 0^+$ disproves the observability inequality~\eqref{eq:gen_frac_obs} and 
proves the Theorem.
\end{proof}

\subsection{Higher dimension}\label{sec:gen_frac_higher_dim}
\Cref{th:gen_frac} can be generalized to take into account the case $\Omega = \set 
R^d \times \set T^{d'}$. Indeed, the Propositions of \cref{sec:technical} are still 
valid in higher dimension. The computations are carried essentially the same way, 
only 
with the added technicalities of the higher dimension, for instance:
\begin{itemize}
 \item in \cref{th:saddle_point_local}, $U$ and $V$ are assumed to be open, bounded 
and convex subset of $\set R^d$ and $\set C^d$ respectively,
 \item in the proof of \cref{th:saddle_point_local}, the change of variables of step 
2 is given by a Morse Lemma with parameter, in the spirit 
of~\cite[Lemma~C.6.1]{hormander_2007},
 \item in \cref{hyp} $\chi$ is chosen to be $C_c^\infty(B(0,\delta))$ (open ball in 
$\set R^d$),
 \item $\rho(|\xi|)$ has to be replaced by $\rho\Big[\big(\sum_i 
\xi_i^2\big)^{1/2}\Big]$ (i.e.\ what happens to be holomorphic in $\xi\in\set C^d$ 
and that is equal to $\rho(|\xi|)$ if $\xi \in \set R^d$),
 \item in all the complex integrals that follows, we integrate against 
$\diff \xi_1 \wedge \dots\wedge \diff \xi_d$,
 \item also, the power of $2\pi h$ in front of the asymptotic expansion is $(2\pi 
h)^{d/2}$ 
(but it does not matter).
\end{itemize}
Then, the construction of the counterexample to the observability inequality if 
$\Omega = \set R^d$ is the same. For the case $\Omega = \set R^d \times \set T^{d'}$, 
we first consider a counter example in $\set R^{d+d'}$, and we periodize the last 
$d'$ components with the method of \cref{sec:heat_bounded}.

\section{Non-null-controllability of the Kolmogorov equation}
\subsection{Introduction}
Now, we look at the Kolmogorov equation~\eqref{eq:kolm_control}.
As for the fractional heat equation, the null-controllability of the Kolmogorov 
equation~\eqref{eq:kolm_control} is 
equivalent to the existence of $C>0$ such that for every solution $g$ 
of\,\footnote{Note that this is the adjoint of the Kolmogorov equation where we 
reversed the time.}
\begin{equation}\label{eq:kolm_adj}
  (\partial_t - v^2 \partial_x - \partial_v^2)g(t,x,v) = 0 \quad 
t\in(0,T),(x,v)\in \Omega
\end{equation}
with Dirichlet boundary conditions if $\Omega_v = (-1,1)$,
\begin{equation}\label{eq:obs_kolm}
 |g(T,\cdot)|_{L^2(\Omega)} \leq C |g|_{L^2((0,T)\times \omega)}.
\end{equation}

As hinted in the introduction, we look 
for counterexamples of the observability inequality among solutions of the adjoint 
of the Kolmogorov equation~\eqref{eq:kolm_adj} of the form $g(t,x,v) = \int_{\set 
R} a(\xi) e^{ix\xi}g_\xi(v)e^{-\lambda_\xi t}\diff\xi$, where $g_\xi(v)$ is the 
first eigenfunction of $-\partial_v^2 - i\xi v^2$ and $\lambda_\xi$ its associated 
eigenvalue. Let us remind that\footnote{We choose the branch of the square root with 
positive real part.} $\lambda_\xi = \sqrt{-i\xi}$ if $\Omega_v = \set R$, and is 
close to $\sqrt{-i\xi}$ if $\Omega_v = (-1,1)$.

We remark that apart from the $g_\xi(v)$ term, those solutions have the same form as 
solutions of the rotated fractional heat equation $(\partial_t+ 
\sqrt{-i}(-\Delta)^{1/4})g = 0$. So, the strategy is to prove the same estimates we 
proved for the rotated fractional heat equation, but with some uniformity in the 
parameter $v$. Since the computations are essentially the same, we only tell what we 
need to care about in comparison with the rotated fractional heat equation, but we 
do not give the full details of the computations again.

\subsection{The Kolmogorov equation with unbounded velocity}\label{sec:kolm_R}
\begin{proof}[Proof of \cref{th:th_kolm} with $\Omega_x= \Omega_v = \set 
R$]In the case $\Omega_v = \set R$, the first eigenfunction of $-\partial_v^2 -i\xi 
v^2$ is  $g_\xi(v) = e^{ - \ssqrt{-i\xi}\s v^2\!/2}$ with eigenvalue $\lambda_\xi = 
\sqrt{-i\xi}$. Without loss of generality, we may assume $\omega_x = \set 
R\setminus [-\eta,\eta]$ (let us remind that $\omega = \omega_x\times \set R$). 

Thus, we consider the function $g_h\colon \set R_+\times\set R^2\to \set C$ defined by
\begin{equation}
\begin{aligned}
 g_h(t,x,v) &=  
 h\int_{\set R} \chi(h\xi-\xi_0) e^{ix\xi -(h\xi-\xi_0)^2\!/2h 
-\ssqrt{-i\xi}\s (v^2\!/2 +t)} \diff \xi\\
 &= \int_{\set R} \chi(\xi-\xi_0) e^{ix\xi/h -(\xi-\xi_0)^2\!/2h 
-\ssqrt{-i\xi/h}\s (v^2\!/2 +t)} \diff \xi.
\end{aligned}
\end{equation}
Since every $(t,x,v)\mapsto e^{ix\xi-\ssqrt{-i\xi}\s(t+ v^2\!/2)}$ is a generalized 
solution to the Kolmogorov equation~\eqref{eq:kolm_adj}, the function $g_h$ is also 
solution to the Kolmogorov equation. We also remark that $g_h(0,\cdot,\cdot)\in 
L^2(\set R^2)$. Notice that these solutions are of the form $\tilde 
g_h(t+v^2\!/2,x)$, 
where $\tilde g_h$ is solution to the ``rotated fractional heat equation'' 
$(\partial_t+ \sqrt{-i}(-\Delta_x)^{1/4})\tilde g_h(t,x)=0$. Thus we have asymptotic 
expansion 
on $\tilde g_h$ similar to~\eqref{eq:lower_coherent} and~\eqref{eq:upper_coherent}.

For $t\geq 0$, $h>0$ and $\Re(\xi)>0$, we define $\rho_{t,h}(\xi) = -t\sqrt{-i\xi}$. 
Thus, with any $K>0$, and with $X = [0,T+1/2]$, \cref{hyp} holds. Thus, with the 
notations of \cref{hyp} and Eq.~\eqref{eq:defI}, we have for $h$ small enough
\[
 g_h(t,x,v) = I_{t+v^2\!/2,h,1}(x).
\]
Moreover, still with the notations of \cref{hyp}, we have for some $C>0$ and for 
every $0<h<1$, $C^{-1}h^{1/2} \leq \epsilon(h) \leq C h^{1/2}$.

Thus, according to \cref{th:ighv_asy} (or more precisely~\cref{rk:ighv_asy}), there 
exist $C,c>0$ such that for every $h>0$ small enough, $0\leq t \leq T$, $|x|$ small 
enough (say $|x|<\eta'$) and $|v|<1$
\begin{align}
 |g_h(t,x,v)| &\geq ce^{-x^2\!/2h -Ch^{-1/2}}.\label{eq:lower_kolm_R}\\
\intertext{Moreover, assuming $K$ large enough and choosing $X = \set R_+$, 
according to \cref{th:ighv_upper} (see also \cref{rk:ighv_upper}), there exist 
$C,c>0$ 
such that for every $h>0$ small enough, $|x|>\eta$, $t\geq 0$ and $v\in \set R$
}
 |g_h(t,x,v)| &\leq C|x|^{-2}e^{-c/h - c(t+v^2\!/2)h^{-1/2}},\label{eq:upper_kolm_R}
\end{align}
So, integrating these estimates, we have
\begin{equation}\label{eq:upper_kolm_R_bis}
 |g_h|_{L^2([0,T]\times \omega)} \leq Ce^{-c/h}
\end{equation}
and
\begin{equation}\label{eq:lower_kolm_R_bis}
 |g_h(T,\cdot,\cdot)|_{L^2(\Omega)} \geq|g_h(T,\cdot,\cdot)|_{L^2(|x|<\eta', 
|v|<1)}\geq ce^{-Ch^{-1/2}}.
\end{equation}

Taking again $h\to 0$ disproves the observability inequality and proves the Theorem.
\end{proof}

For the Kolmogorov equation with $\Omega_x = \set T$ and $\Omega_v = \set R$, we 
define $\per{g_h}(t,x,v) = \sum_{k\in\set Z} g_h(t,x+2\pi k,v)$, and as 
in \cref{sec:heat_bounded}, all but the term for $k=0$ are $\bigO(e^{-c/h})$. We 
let the careful reader work out the details.

\subsection{The Kolmogorov equation with non-rectangular control 
domain}\label{sec:kolm_b_g}
\begin{proof}[Proof of \cref{th:th_kolm_general} in the case $\Omega_x = 
\Omega_v = \set R$]
\begin{figure}
 \begin{minipage}{0.5\textwidth}
\centering
\begin{tikzpicture}[baseline = (O), scale=2.5]
 \fill[green!40, even odd rule] (-1,0.6) rectangle (1,-0.45) (0.3,0.1) 
circle[radius=0.4];
 \node (O) at  (0,0){};
 \draw (0.3,0.1) circle[radius=0.4];
 \filldraw[fill=blue!50] (0.2,-0.27) rectangle (0.4,0.27);
 \draw[very thick, blue] (0.3,-0.3) -- (0.3,0.3);
 \draw[dashed, blue!50] (0.3,-0.4)node[below, blue]{$x_0$} -- 
(0.3,-0.3);
 \draw[densely dashed, blue!50] (0,-0.3) node[blue, left]{$-a$} -- (0.3,-0.3)
 (0,0.3) node[blue, left]{$a$} -- (0.3,0.3);
 \draw[->] (-1.1,0) -- (1.1,0) node[above]{$x$};
 \draw[->] (0,-0.5) -- (0,0.7) node[right]{$v$};
 \node at (-0.6,0.4){$\omega$};
\end{tikzpicture}%
 \end{minipage}%
 \begin{minipage}[c]{0.4\textwidth}
\caption{In green, the control domain $\omega$. If there is a 
vertical line, symmetrical with respect to $\{v=0\}$, that does not intersect 
$\bar \omega$ (in dark blue), for every $a'<a$, there exists a rectangle of the 
form $\{|x-x_0|<b, -a'<v<a'\}$ that does not intersect $\bar\omega$ (in 
lighter blue).}\label{fig:omega}
\end{minipage}
\end{figure}
If $0<a'<a$, there exists $b>0$ such that the 
rectangle  $R=\{|x-x_0|<b, |v|<a'\}$ does not intersect $\overline \omega$ (see 
\cref{fig:omega}). Since the equation is invariant by translation in the $x$ 
direction, we may assume without loss of generality that $x_0 = 0$. 

We will use the same functions 
$g_h$ as in the previous proof. But while we used only the terms of order 
$h^{-1}$ in the exponent of estimate of \cref{th:ighv_asy}, we will 
now use the next term. More precisely, according to \cref{th:ighv_asy} with $X = 
[0,T+a^2\!/2]$ and $\rho_{t,h}(\xi) = -t\ssqrt{-i\xi}$, we have uniformly 
in $|x|$ small enough, $0\leq t \leq T$ and $|v|<a$:
\begin{equation}\label{eq:asy_g_precise}
 g_h(t,x,v) = I_{t+v^2\!/2,h,1}(x) = \sqrt{2\pi h}\s 
e^{\phi(t,x,v)/h}\big(1+\bigO_h(\sqrt h)\big)
\end{equation}
with
\begin{equation}
 \phi(t,x,v) = ix\xi_0 -\frac{x^2}2 -\sqrt{-i\xi_0-x}\s\left(t+\frac{v^2}2\right) 
h^{1/2} + \bigO_h(h).
\end{equation}

The idea is that, when computing $\int_{\Omega} |g_h(T,x,v)|^2\diff x\diff v$, the 
dominant part of this integral is around $x=v=0$, and when computing 
$\int_{[0,T]\times \omega} |g_h(t,x,v)|^2\diff t\diff x\diff v$, the dominant part 
is around $t = 0, x = x_0 = 0$ and $v = a$ or $-a$. So, noting ${c_0 = 
\Re(\ssqrt{-i\xi_0}) >0}$ 
and ignoring the error terms for the moment, we have
\begin{equation}
\int_{\Omega} |g_h(T,x,v)|^2\diff x\diff v 
\approx 2\pi h\int_{\substack{|x|<\epsilon\\|v|<\epsilon}} 
e^{-x^2\!/h - c_0(2T+v^2)/\sqrt h} \diff x \diff v \approx 
Ch^{7/4}e^{-2Tc_0/\sqrt h} 
\end{equation}
and
\begin{multline}
\int_{[0,T]\times\omega} |g_h(t,x,v)|^2\diff t\diff x\diff v 
\approx 2\pi h
\int_{\smash[b]{\substack{|x|<\epsilon\\a<|v|<a+\epsilon\\t<\epsilon}}}
e^{-x^2\!/h - c_0(2t+v^2)/\sqrt h} \diff t \diff x \diff v \approx 
Ch^{5/2}e^{-c_0a^2\!/\sqrt h}.
\end{multline}
So, if $2T<a^2$, the observability inequality cannot hold. 
Now, let us rigorously prove this. Let $\epsilon>0$ and $T=a'^2\!/2-\epsilon$.  
We have $\Re(\sqrt{-i\xi_0-x}) = c_0 +\bigO_x(x)$. So, for $x$ small enough, 
say, $|x|<\delta_{\xi_0}$, we have 
\begin{equation}
c_0-\epsilon \leq\Re\left(\sqrt{-i\xi_0-x}\right) \leq c_0+\epsilon.
\end{equation}

So, we have locally uniformly in $|x|<\delta_{\xi_0}$, $t\geq 0$ and $v\in\set R$ :
\begin{equation}\label{eq:lower_phase}
 \Re(\phi(t,x,v)) \geq -\frac{x^2}2 -(c_0+\epsilon) 
\left(t+\frac{v^2}2\right) h^{1/2} -\bigO_h(h)
\end{equation}
and
\begin{equation}\label{eq:upper_phase}
 \Re(\phi(t,x,v)) \leq -\frac{x^2}2 -(c_0-\epsilon) 
\left(t+\frac{v^2}2\right) h^{1/2} + \bigO_h(h).
\end{equation}

Now, let us get a lower bound for the left-hand side of the observability 
inequality~\eqref{eq:obs_kolm}. We have:
\begin{align}
 |g_h(T,\cdot,\cdot)|^2_{L^2(\Omega)} 
 &\geq |g_h(T,\cdot,\cdot)|^2_{L^2(|x|<b, |v|<a')}
 \intertext{and thanks to the asymptotic of Eq.~\eqref{eq:asy_g_precise}:}%
 |g_h(T,\cdot,\cdot)|^2_{L^2(\Omega)} 
 &\geq 2\pi h \int_{\substack{|x|<b\\|v|<a'}}
 e^{2\Re(\phi(T,x,v))/h} \diff x\diff v\big(1+\bigO(\sqrt h)\big)\\
 \intertext{and with the lower bound above (Eq.~\eqref{eq:lower_phase}):}
 |g_h(T,\cdot,\cdot)|^2_{L^2(\Omega)} 
 &\geq 2\pi h\s e^{\bigO_h(1)} \int_{\substack{|x|<b\\|v|<a'}}
 e^{-x^2\!/h-(c_0+\epsilon)(2T+v^2)/\sqrt h} \diff x\diff v\big(1+\bigO(\sqrt 
h)\big).
\end{align}
The integral in $x$ is
\begin{equation}
 \int_{|x|<b} e^{-x^2\!/h}\diff x = \sqrt{\pi h}+\bigO(e^{-c/h})
\end{equation}
while the integral in $v$ is
\begin{equation}
 \int_{|v|<a'} e^{-(c_0+\epsilon)v^2\!/\sqrt h}\diff v = 
\sqrt{\frac{\pi}{c_0+\epsilon}}h^{1/4} + \bigO\big(e^{-c/\sqrt h}\big).
\end{equation}
So, we have
\begin{align}
 |g_h(T,\cdot,\cdot)|^2_{L^2(\set R^2)} 
 &\geq c\frac{2\pi^2}{\sqrt{c_0+\epsilon}}
 h^{7/4}e^{-(c_0+\epsilon)2T/\sqrt h}\big(1+\bigO(\sqrt h)\big)\\
 \intertext{and for $h$ small enough:}
 |g_h(T,\cdot,\cdot)|^2_{L^2(\Omega)} 
 &\geq c h^{7/4}e^{-(c_0+\epsilon)2T/\sqrt h}\label{eq:lower_kolm_gen}.
\end{align}

Now, let us bound the right-hand side of the observability 
inequality~\eqref{eq:obs_kolm}. Let us remind that $\omega$ is a subset of 
$\Omega\setminus\{|x|<b, |v|<a'\}$. Let $a''>a'$, that will be 
chosen large enough afterwards. We define 
\begin{align}
\omega_0 &= \{|x|<\delta_{\xi_0}, a'<|v|<a''\}\label{eq:omega0}\\
\omega_1 &=\{|x|\geq\delta_{\xi_0}\}\label{eq:omega1}\\
\omega_{2} &= \{|x|<\delta_{\xi_0}, |v|>a''\}\label{eq:omega2}.
\end{align}
With these definitions, if $\delta_{\xi_0}<b$, we have $\omega\subset 
\omega_0\cup\omega_1\cup\omega_2$. So,
\begin{equation}
 |g_h|_{L^2([0,T]\times \omega)}^2 \leq 
|g_h|_{L^2(|0,T]\times\omega_{0})}^2 + 
|g_h|_{L^2(|0,T]\times\omega_{1})}^2 + 
|g_h|_{L^2(|0,T]\times\omega_{2})}^2
\end{equation}

First, according to \cref{th:ighv_upper_simple}, we have for every $t\geq 0, v\in 
\set R$ and $x\in \set R$,
\begin{equation} 
 |g_h(t,x,v)| \leq Ce^{-c'(t+v^2\!/2)h^{-1/2}}
\end{equation}
So, integrating this estimate, we have
\begin{equation}
 |g_h|_{L^2([0,T]\times \omega_2)}^2 = \bigO\left( \int_{|v|\geq a''} 
e^{-c'v^2\!/\sqrt h}\diff v\right) = \bigO\left(e^{-c'a''^2\!/\sqrt h}\right).
\end{equation}
We choose $a''$ large enough so that $c'a''^2 > (c_0-\epsilon)a'^2$. That way, we 
have
\begin{equation}\label{eq:up1}
 |g_h|_{L^2([0,T]\times \omega_2)}^2 =  \bigO\left(e^{-(c_0-\epsilon)a'^2\!/\sqrt 
h}\right).
\end{equation}
We have already seen in \cref{sec:kolm_R} that
\begin{equation}\label{eq:up2}
 |g_h|_{L^2([0,T]\times \omega_1)}^2 = \bigO(e^{-c/h}).
\end{equation}
Finally, thanks to Eq.~\eqref{eq:asy_g_precise} with upper 
bound~\eqref{eq:upper_phase}, we have uniformly in $0\leq t\leq T$, 
$|x|<\delta_{\xi_0}$ and $a'<|v|< a''$
\begin{equation}
 |g_h(t,x,v)|^2 \leq  2\pi h\s e^{-x^2\!/h - (c_0-\epsilon)(2t+v^2)/\sqrt h 
+\bigO_h(1)}.
\end{equation}
So, we have
\begin{equation}\label{eq:up3}
 |g_h|_{L^2([0,T]\times \omega_0)}^2 
=\bigO\left(\int_{a'<|v|<a''}e^{-(c_0-\epsilon)v^2\!/\sqrt h}\diff v\right) = 
\bigO\left(e^{-(c_0-\epsilon)a'^2\!/\sqrt h}\right).
\end{equation}

So, putting the three upper bounds \eqref{eq:up1} \eqref{eq:up2} and \eqref{eq:up3} 
together, we have
\begin{equation}\label{eq:upper_kolm_gen}
 |g_h|_{L^2([0,T]\times \omega)}^2 = \mathcal O\left(e^{-(c_0-\epsilon)a'^2\!/\sqrt 
h}\right).
\end{equation}

Let us assume that $(c_0-\epsilon)a'^2 > 2T(c_0+\epsilon)$. Then, considering the 
previous upper bound~\eqref{eq:upper_kolm_gen} and the lower 
bound~\eqref{eq:lower_kolm_gen}, and taking $h\to 0$ disproves the observability 
inequality. So, the Kolmogorov equation is not null-controllable in time 
$T<\frac{c_0-\epsilon}{c_0+\epsilon} a'^2\!/2$. This is true for every $a'<a$ and 
$\epsilon>0$, so the Kolmogorov equation is not null-controllable in time 
$T<a^2\!/2$.
\end{proof}

The case $\Omega_x = \set T$, $\Omega_v = \set R$ is similar. We look at 
$\per{g_h}(t,x,v) = 
\sum_{k\in\set Z} g_h(t,x+2\pi k,v)$. In this sum, as in \cref{sec:heat_bounded}, 
only the term for $k=0$ matters, as the other are 
$\bigO(e^{-c/h})$.

\subsection{The Kolmogorov equation with bounded velocity}\label{sec:kolm_bounded}
To treat the Kolomogorov equation with $\Omega_v = (-1,1)$, we need some information 
on the first eigenfunction $g_\xi$ of $-\partial_v^2 - i\xi v^2$ with Dirichlet 
boundary 
conditions on $(-1,1)$, and with associated eigenvalue $\lambda_\xi = \sqrt{-i\xi} + 
\rho_{\xi}$. Moreover, as we will use 
Theorems~\ref{th:ighv_asy}--\ref{th:ighv_upper_simple}, we also need 
some analycity in $\xi$. We will denote $\tilde g_{\tilde \xi}$ the 
first\footnote{``First'' in 
the sense that it is the analytic continuation in $\tilde\xi$ of the first 
eigenfunction 
of $-\partial_v^2 + (\tilde \xi v)^2$ for $\tilde \xi \in \set R_+$, assuming it 
exists.} eigenfunction of $-\partial_v^2 + (\tilde \xi v)^2$, and $\tilde 
\lambda_{\tilde\xi} = \tilde\xi + \tilde\rho_{\tilde \xi}$ the associated eigenvalue, 
so that, with $\tilde\xi = \ssqrt{-i\xi}$, we have $g_\xi = \tilde g_{\tilde \xi}$ 
and 
$\rho_\xi = \tilde \rho_{\tilde\xi}$, when this is defined.

In an article on the Grushin equation \cite[Section 4]{koenig_2017} we proved 
that $\tilde \rho_{\tilde \xi}$ and $\tilde g_{\tilde\xi}$ exist if 
$\Re(\tilde\xi)>0$ 
and $|\tilde\xi|>r(\lvert\arg(\tilde\xi)\rvert)$ for some non-decreasing function 
$r:(0,\pi/2)\to \set R_+$. We also proved the next two theorems.
\begin{theorem}[Theorem 22 and Remark 23 of 
\cite{koenig_2017}]\label{th:asymptotic}
Let $0<\theta<\pi/2$. With $\tilde \rho_{\tilde \xi}$ defined above, we have 
\[\tilde\rho_{\tilde\xi} \sim \frac{4}{\sqrt \pi} {\tilde\xi}^{3/2}e^{-\tilde\xi}\]
in the limit $|\tilde\xi|\to \infty, \lvert\arg(\tilde\xi)\rvert < \theta$.
\end{theorem}
\begin{proposition}[Proposition 25 of \cite{koenig_2017}]\label{th:est_agmon}
Let $\tilde g_{\tilde\xi}$ be defined above and normalized by $\tilde 
g_{\tilde\xi}(0) = 1$ (instead of $|\tilde g_{\tilde\xi}|_{L^2} = 1$). Let 
$0<\theta<\pi/2$ and $\epsilon>0$. We have 
for all $v\in(-1,1)$ and $|\tilde\xi|>r(\theta)$, 
$\lvert\arg(\tilde\xi)\rvert<\theta$:
\[|e^{(1-\epsilon)\tilde\xi v^2\!/2}\tilde g_{\tilde\xi}(v)| \leq 
C_{\epsilon,\theta}.\]
\end{proposition}

\Cref{th:asymptotic} gives us all we need to know on the eigenvalue, 
while \cref{th:est_agmon} gives us an upper bound on the eigenfunction. We will 
also need the following lower bound, that we prove in \cref{sec:agmon_lower}.
\begin{proposition}\label{th:lower_agmon}
 Let $0<\theta<\pi/2$ and $\epsilon>0$. We normalize $\tilde g_{\tilde\xi}$ again by 
$\tilde 
g_{\tilde\xi}(0) = 1$ and define $\tilde u_{\tilde \xi}(v) = e^{ \tilde\xi 
v^2\!/2}\tilde g_{\tilde\xi}(v)$. Then 
$\tilde u_{\tilde\xi}(v)$ converges exponentially fast to $1$, as $|\tilde\xi|\to 
\infty$, 
$\lvert\arg(\tilde\xi)\rvert<\theta$, this convergence being uniform in 
$|v|<1-\epsilon$.
\end{proposition}

With this, we know all we need to adapt the proof of the non-null-controllability 
of the 
Kolmogorov equation with $\Omega_v = \set R$ to the case of $\Omega_v = (-1,1)$.
\begin{proof}[Proof of \cref{th:th_kolm} with $\Omega_v = (-1,1)$]
We start with the case $\Omega_x = \set R$.

\step{Step 1: construction of the counterexample to the observability inequality.}
 The counterexample we build to the observability inequality~\eqref{eq:obs_kolm} 
is basically the same as in the case $\Omega_v = \set R$, only with the added 
corrections to the eigenvalues and eigenfunctions. We define $g_h(t,x,v)$ for $t\geq 
0$, $x\in \set R$, $v\in(-1,1)$ and $h>0$ small enough by:
\begin{equation}\label{eq:def_g_kolm_b_wo_recentering}
 g_h(t,x,v) = \int_{\set R}\chi(\xi-\xi_0) 
 e^{-ix\xi/h - (\xi-\xi_0)^2\!/2h -\lambda_{\xi/h}t} g_{\xi/h}(v) \diff \xi,
\end{equation}
where $\xi_0>0$ and $\chi\in C_c^\infty(\set R)$ are chosen as follows.

First note that according to the discussion at the top of this subsection, 
$\lambda_\xi$ and $g_\xi$ are defined and holomorphic with respect to $\xi \in \set 
C$ such that $\lvert\arg(\xi)\rvert < 3\pi/8$ (for instance) and $|\xi|$ 
large enough. Let then $K>0$ be large enough so that for any $\xi\in 
\Cc\coloneqq\{\Re(\xi)>K,|\Im(\xi)|<K^{-1}\Re(\xi)\}$, $\lambda_\xi$ and $g_\xi$ are 
defined and holomorphic with respect to $\xi\in\Cc$. Finally, let $\xi_0>0$ and 
$\chi$ as in \cref{hyp} (see \cref{fig:chgt_path}). With these choices, $g_h$ is 
well-defined for $0<h\leq 1$, 
$t\geq 0$, $x\in \set R$ and $v\in (-1,1)$.

We remark that each function $(t,x,v)\mapsto e^{-ix\xi-\lambda_{\xi}t}g_\xi(v)$ is 
solution to the Kolmogorov equation~\eqref{eq:kolm_adj}. So $g_h$ is solution of the 
Kolmogorov equation.

\begin{figure}
 \begin{center}
  \begin{tikzpicture}[baseline = (O), scale=0.8]
 \node (O) at (0,0){};
 \fill[red, opacity = 0.7] (20:1) -- (20:3.7) arc(20:-20:3.7) -- (-20:1) -- cycle;
 \draw[red,thick] (-20:3.7) -- (-20:1) -- (20:1) -- (20:3.7);
 \draw [->] (-1,0) -- (4,0);
 \draw [->] (0,-3) -- (0,1);
 \node (C) at (2,-0.2) {};
 \draw[red!70!black, thick,->] (1.5,-1) node[below]{$\Cc$} -- (C);
 \draw[thick,->] (1.7,0) arc (0:20:1.7) node[pos = 0.5, right]{$\arctan(1/K)$};
 \draw[thick] (0.94,-0.2) -- (0.94,0.2) node[above]{$K$};
\end{tikzpicture}%
\kern1em%
  \begin{tikzpicture}[baseline = (O), scale=0.8]
 \node (O) at (0,0){};
 \draw[thick] (75:1.03) -- (75:0.3)
   arc[radius=0.3, start angle= 75, end angle = -75] -- (-75:3.5);
 \path[fill=gray!40] (75:1.03) -- (75:0.3)
   arc[radius=0.3, start angle= 75, end angle = -75] -- (-75:3.5)
   arc[radius=3.5, start angle=-75, end angle=16.6]-- cycle;
 \draw [->] (-1,0) -- (4,0);
 \draw [->] (0,-3) -- (0,1);
 
 \begin{scope}[rotate=-45, scale = 0.95]
 \node (C) at (1.5,-0.3){};
 \node(C1) at(1.015,0) {};
 \node(C2) at (2,0){};
 \fill[red, opacity = 0.7] (10:1) -- (10:3.7) arc(10:-10:3.7) -- (-10:1) -- cycle;
 \draw[red,thick] (-10:3.7) -- (-10:1) -- (10:1) -- (10:3.7);
 \node (C) at (2.5,0) {};
 \draw[red!70!black, thick,->] (2.2,1) node[above]{$\sqrt{-i\Cc}$} -- (C.center);
 \end{scope}
\end{tikzpicture}%
 \end{center}
 \caption{Left figure: in red, shape of $\Cc$ for some $K$. Right figure: if $\tilde 
\xi$ is in the gray domain, the eigenvalue
$\tilde\rho_{\tilde \xi}$ and the eigenfunction $\tilde g_{\tilde \xi}$ of 
$-\partial_v^2 + (\tilde\xi v)^2$ are defined. Thus, the eigenvalue $\lambda_\xi$ and 
eigenfunction $g_\xi$ of $-\partial_v^2 + i\xi v^2$ are defined for $\xi\in\Cc$ if 
$\sqrt{-i\Cc}$ lies inside the gray domain. } 
\label{fig:chgt_path}
\end{figure}

\step{Step 2: estimates on $g_h$.} Note that \cref{th:asymptotic} 
and \cref{th:est_agmon,th:lower_agmon} (with the choice $\epsilon = 1/2$) translate 
respectively into the estimates:
\begin{align}
 |e^{-t\rho_{\xi}}-1| &\leq Ce^{-c\ssqrt{|\xi|}} &&\text{for } \xi\in\Cc\text{ and } 
0\leq t\leq T\label{eq:est_rho}\\
 |e^{\ssqrt{-i\xi}\s v^2\!/4}g_{\xi}(v)| &\leq C 
&&\text{for } \xi\in \Cc \text{ and } |v|<1\label{eq:upper_g}\\
 |e^{\ssqrt{-i\xi}\s v^2\!/2} g_{\xi}(v) - 1| &\leq 
Ce^{-c\ssqrt{|\xi|}}
   &&\text{for } \xi\in \Cc \text{ and } |v|<1/2.\label{eq:est_g}
\end{align}
for some $C,c>0$.

\step{Step 2a: lower bound on $g_h$.}
We want to write $g_h(t,x,v)$ in the form of Eq.~\eqref{eq:defI}. Let $X = 
[0,T+1/2]$, and for $t\in X$, $0<h\leq 1$ and $\xi\in \Cc$, let $\rho_{t,h}(\xi) = 
-t\ssqrt{-i\xi}$. Finally, for $0<h\leq1$, $t\geq 0$, $v\in(-1,1)$ and $\xi\in\Cc$, 
let
\begin{equation}\label{eq:def_delta}
 \delta_{h,t,v}(\xi) \coloneqq e^{\ssqrt{-i\xi/h}\s v^2\!/2}g_{\xi/h}(v) 
e^{-t\rho_{\xi/h}}.
\end{equation}
Then, according to the definition of $g_h$ 
(Eq.~\eqref{eq:def_g_kolm_b_wo_recentering}),
\begin{equation}\label{eq:def_g_with_delta}
 g_h(t,x,v) =
 \int_{\set R}\chi(\xi-\xi_0) 
 e^{-ix\xi/h - (\xi-\xi_0)^2\!/2h -\ssqrt{-i\xi/h}\s(t+v^2\!/2)}
 \delta_{h,t,v}(\xi) \diff \xi = I_{t+v^2\!/2,h,\delta}(x).
\end{equation}
Moreover, according to estimates Eq.~\eqref{eq:est_rho} and~\eqref{eq:est_g}, we 
have, for some $C,c>0$:
\begin{equation}\label{eq:asy_delta}
 |\delta_{h,t,v}(\xi)-1| \leq Ce^{-ch^{-1/2}} \quad \text{for } \xi\in \Cc,\ 
|v|<1/2 \text{ and } 0<h\leq 1.
\end{equation}

So, according to \cref{th:ighv_asy}, there exists $C,c>0$ such that for every 
$t\in[0,T]$, $|v|<1/2$, $x$ small enough and $h$ small enough,
\begin{equation}
  |g_h(t,x,v)| \geq ce^{-x^2\!/2h -Ch^{-1/2}}.\label{eq:lower_kolm_b}
\end{equation}

\step{Step 2b: upper bound on $g_h$.} The estimate~\eqref{eq:est_g} does not extend 
up to the boundary. Thus, we have to use the less precise upper 
bound~\eqref{eq:upper_g}. To this end, we define $\tilde \delta_{h,t,v}(\xi) = 
e^{-\ssqrt{-i\xi/h}\s v^2\!/2}\delta_{h,t,v}(\xi)$. Then, according to the definition 
of $g_h$ (Eq.~\eqref{eq:def_g_kolm_b_wo_recentering}) and $\delta$ 
(Eq.~\eqref{eq:def_delta}),
\begin{equation}\label{eq:def_g_with_delta2}
 g_h(t,x,v) =
 \int_{\set R}\chi(\xi-\xi_0) 
 e^{-ix\xi/h - (\xi-\xi_0)^2\!/2h -\ssqrt{-i\xi/h}\s(t+v^2\!/4)}
 \tilde\delta_{h,t,v}(\xi) \diff \xi = I_{t+v^2\!/4,h,\tilde\delta}(x).
\end{equation}
Moreover, according to estimates~\eqref{eq:est_rho} and~\eqref{eq:upper_g}, there 
exist $C,c>0$ such that
\begin{equation}\label{eq:upper_delta}
  |\tilde \delta_{h,t,v}(\xi)| \leq C \quad \text{for } \xi\in \Cc,\ 
|v|<1 \text{ and } 0<h\leq 1.
\end{equation}

So, according to \cref{th:ighv_upper}, there exist $C,c>0$ such that for every 
$t\in[0,T]$, $|v|<1$, $|x|>\eta$ (where we assume without loss of generality $\omega 
= \{|x|>\eta\}\times (-1,1)$) and $h$ small enough,
\begin{equation}
  |g_h(t,x,v)| \leq \frac{C}{|x|^2}e^{-c/h}.\label{eq:upper_kolm_b}
\end{equation}

\step{Step 3: conclusion.}
From this point on, the proof is the same as in \cref{sec:kolm_R}: 
integrating the estimates~\eqref{eq:lower_kolm_b} and~\eqref{eq:upper_kolm_b} proves 
that $g_h$ is a counterexample to the observability inequality~\eqref{eq:obs_kolm} in 
the case $\Omega = \set R \times (-1,1)$ and $\omega = \omega_x\times (-1,1)$.

In the case $\Omega_x = \set T$, we again look at the periodic version of $g_h$, 
that is $\per{g_h}(t,x,v) = \sum_{k\in \set Z} g_h(t,x+2\pi k,v)$. As 
in \cref{sec:heat_R} (and~\ref{sec:kolm_R}), $\per{g_h}$ is a solution to the 
Kolmogorov equation, and it is a counterexample to the observability inequality.
\end{proof}

The proof of \cref{th:th_kolm_general} in the case $\Omega_v = (-1,1)$ is 
similar to the case $\Omega_v = \set R$ with the adaptation of the previous proof. 
Let us just sketch it.

We choose $a'<a$ and $b>0$ as in \cref{sec:kolm_b_g}. We consider the functions $g_h$ 
of the previous proof (Eq.~\eqref{eq:def_g_kolm_b_wo_recentering}). 

We compute the next order in the estimate~\eqref{eq:lower_kolm_b}. With 
\cref{th:ighv_asy}, we can prove that locally
uniformly in $v\in(-1,1)$, $|x|$ small enough and $t> 0$,
\begin{equation}\label{eq:lower_kolm_b_g}
 g_h(t,x,v) = \sqrt{2\pi h}\s e^{ix\xi_0/h - x^2\!/2h - 
\sqrt{\xi_0+ix}(t+v^2\!/2)/\sqrt h + \bigO_h(1)}\left(1+\bigO(\sqrt 
h)\right).
\end{equation}
Also, thanks to \cref{th:ighv_upper}, we prove that uniformly in $|x|>b$, $t>0$ 
and $v\in(-1,1)$, 
\begin{equation}\label{eq:upper_kolm_b_g}
 g_h(t,x,v) = \bigO(|x|^{-2}e^{-c/h}).
\end{equation}
We choose $a'<a''<1$ and we define $\omega_0$, $\omega_1$ and $\omega_2$ as in 
equations~\eqref{eq:omega0},~\eqref{eq:omega1} and~\eqref{eq:omega2} (the 
$\delta_{\xi_0}$ is the same in the cases $\Omega_v = (-1,1)$ and $\Omega_v = \set 
R$).

With the estimate~\eqref{eq:lower_kolm_b_g}, we can prove an estimate similar to the 
lower bound~\eqref{eq:lower_kolm_gen}. We can also prove an upper bound similar 
to~\eqref{eq:up3}. With the estimate~\eqref{eq:upper_kolm_b_g}, we can prove an 
estimate similar to~\eqref{eq:up2}. And with the help of 
\cref{th:est_agmon} to manage the terms for $|v|>a''$, we can prove an 
upper bound similar to~\eqref{eq:up1}. The rest of the proof is a copy-paste.

\appendix
\section{Other equations}\label{sec:other_eq}
In this appendix, we explain how we can use the method of \cref{sec:gen_frac} to 
prove the lack of null-controllability for some other equations.
\subsection{Fractional Schrödinger equations}
Let $0\leq \alpha<1$. If we consider $\rho(\xi) = i\xi^\alpha$ (defined e.g.\ for 
$\Re(\xi)\geq 0$), the hypotheses of \cref{th:gen_frac} hold. Thus, we have:
\begin{corollary}
 Let $0\leq \alpha<1$. Let $T>0$ and $\omega$ be a strict open subset of $\set R$. 
The fractional Schrödinger equation $(\partial_t + i(-\Delta)^{\alpha/2})f(t,x) = 
\mathds 1_\omega u(t,x)$, $t\geq 0$, $x\in\set R$ is not null-controllable on 
$\omega$ in time $T$.
\end{corollary}
Since $i(-\Delta)^{\alpha/2}$ generates a strongly-continuous \emph{group} of 
bounded operators on $L^2(\set T)$, it seems likely that this corollary can be 
extended to any riemannian manifold (and not only $\set R^d \times \set 
T^d$).\footnote{The fact that $i(-\Delta)^{\alpha/2}$ generates a strongly continuous 
group also implies that null-controllability is equivalent to 
\emph{exact}-controllability~\cite[Th.~2.41]{coron_2007}.} But this is outside 
the scope of this article. Also, we conjecture that the threshold $\alpha<1$ is 
optimal. Indeed, it seems that if $\omega$ satisfies the Geometric Control Condition 
of Bardos, Lebeau and Rauch~\cite{bardos_1992}, then the methods 
of~\cite{lebeau_1992} could be used to prove exact controllability in any time if 
$\alpha>1$ (see~\cite{filho_2020} for the case $\alpha\geq 2$) and with a minimal 
time if $\alpha = 1$ (better known as ``the half-wave equation''), but this is again 
outside the scope of this article. See also~\cite{biccari_2018} for a variant 
of the fractional Schrödinger equation.

\subsection{Another Kolmogorov-type equation}
Our method can also be used to prove that the Kolmogorov-type equation $(\partial_t 
-\partial_v^2 + v\partial_x)f(t,x,v) = \mathds 1_\omega u(t,x,v)$ is not 
null-controllable on vertical bands (notice that is is Eq.~\eqref{eq:kolm_control} 
where we replaced $v^2$ by $v$).

\begin{theorem}\label{th:kolm_v}
 Let $\Omega = (0,+\infty) \times \Omega_x$, 
and $\Omega_x = \set R$ or $\set T$. Let $\omega_x$ be a strict open subset of 
$\Omega_x$ and $\omega = \omega_x \times (0,+\infty)$, and let $T>0$. Then the 
equation
\begin{equation}\label{eq:kolm_control_v}
 \begin{aligned}
  (\partial_t + v \partial_x - \partial_v^2)f(t,x,v) &= \mathds 1_\omega 
u(t,x,v) &&\quad t\in [0,T], (x,v) \in \Omega\\
  f(t,x,0) &= 0 && \quad t\in[0,T], x\in \Omega_x\\
  f(0,x,v) &= f_0(x,v) &&\quad (x,v) \in \Omega
 \end{aligned}
\end{equation}
is not null-controllable on $\omega$ in time $T$.
\end{theorem}

\begin{proof}[Sketch of the proof]
We consider $\Ai$ the standard Airy function (see for instance~\cite[Ch.~9]{olver_2019}). 
Let $-\mu_0$ the first zero of $\Ai$. We denote $\lambda_0 = e^{i\pi/3}\mu_0$. For 
$\xi>0$, let $u_\xi\colon \set R \to \set C$ defined by $u_\xi(v) = 
\Ai(\xi^{1/3}e^{-i\pi/6}v-\mu_0)$. Using the ODE satisfied by $\Ai$ 
(\cite[\S9.2(i)]{olver_2019}), we see that 
$(-\partial_v^2 -i\xi v)u_\xi = \xi^{2/3}\lambda_0 u_\xi$. Moreover, $u_\xi(0) = 
0$, and  according to the asymptotic expansion satisfied by $\Ai$ 
(\cite[\S9.7ii]{olver_2019}), $u_\xi$ decays exponentially at $\infty$, as well as its 
derivatives. So 
$u_\xi$ is an eigenfunction of $-\partial_v^2 -i\xi v$ on $(0,+\infty)$ with 
Dirichlet boundary condition at $v=0$.

Let $\xi_0>0$ and $\chi\in C_c^\infty(-\xi_0,\xi_0)$. For $h>0$  
we consider the function $g_h\colon \set R_+\times \set R \times (0,+\infty) \to \set 
C$ defined by
\begin{equation}
 \label{eq:def_g_kolm_v}
 \begin{aligned}
 g_h(t,x,v) &= 
 h\int_{\set R_+} \chi(h\xi-\xi_0)e^{-(h\xi-\xi_0)^2\!/2h + ix\xi 
-t\lambda_0\xi^{2/3}}u_\xi(v) \diff \xi\\
 &=\int_{\set R_+} \chi(\xi-\xi_0)e^{-(\xi-\xi_0)^2\!/2h + ix\xi/h 
-t\lambda_0\xi^{2/3}h^{-2/3}}u_{\xi/h}(v) \diff \xi.
 \end{aligned}
\end{equation}

Since $u_\xi$ is an eigenfunction of $-\partial_v^2 -i\xi v$, $g_h$ is solution to 
$(\partial_t -  v \partial_x - \partial_v^2)g_h = 0$. Moreover, using the asymptotic 
expansion 
of $\Ai$~\cite[\S9.7(ii)]{olver_2019}, we have uniformly in $v>1$, in the limit 
$|\xi|\to +\infty$, 
$\lvert\arg(\xi)\rvert <\pi/2$ (for instance)
\[
 u_\xi(v) =
\exp\left(-\frac23(e^{-i\pi/6}\xi^{1/3}v-\mu_0)^{3/2}\right)\tilde u_\xi(v)\quad\text{
with}\quad \tilde u_\xi(v) = 
C\xi^{-1/12}v^{-1/4}(1+\bigO(\xi^{-1/2})).
\]
Thus, we can rewrite 
Eq.~\eqref{eq:def_g_kolm_v} as
\begin{equation}
 g_h(t,x,v) = \int_{\set R_+}\chi(\xi-\xi_0) e^{-(\xi-\xi_0)^2\!/2h +ix\xi/h 
+\rho_{t,v,h}(\xi/h)} \tilde u_{\xi/h}(v)\diff \xi
\end{equation}
with
\[
\rho_{t,v,h}(\xi) = -t\lambda_0 \xi^{2/3} 
-\frac23(e^{-i\pi/6}\xi^{1/3}v-\mu_0)^{3/2}.
\]

Thus, $g_h(t,x,v)$ can be written in the form~\eqref{eq:defI}. Moreover, if we 
choose $K>0$, then we can choose $\xi_0>0$ and $\chi\in C_c^\infty(\set R)$ such that 
\cref{hyp} holds with $X =\{(t,v), 0\leq t\leq T, 1\leq v \leq 2\}$. Then, 
\cref{th:ighv_asy} can be used to prove the lower-bound
\[
 |g_h(t,x,v)| \geq ce^{-x^2\!/2h -C h^{-2/3}},\quad t\in[0,T], |x|\text{ small 
enough}, v\in [1,2].
\]

To get an upper-bound, we choose $K$ large enough in \cref{hyp} so that 
$|\tilde u_\xi(v)|\leq C$ for some $C>0$ and every $\xi\in\Cc$ and $v\in(0,+\infty)$. 
Then, with the choice $X = [0,T]\times (0,+\infty)$ in \cref{hyp}, the
\cref{th:ighv_upper} can be used to prove the upper-bound
\[
 |g_h(t,x,v)|\leq \frac C{|x|^2}e^{-c/h-cv^{3/2}}.
\]

As for the Kolmogorov equation~\eqref{eq:kolm_control}, these two estimates prove 
that the observability inequality associated with the control 
problem~\eqref{eq:kolm_control_v} does not hold if $\Omega_x= \set R$. For the case 
$\Omega_x = \set T$, we periodize the solutions as in \cref{sec:heat_bounded}.
\end{proof}

We refer to \cref{sec:kolm_bib} for references related to the 
equation~\eqref{eq:kolm_control}. It seems \cref{th:kolm_v} could be extended to the 
case $\Omega = \Omega_x \times (a,b)$, as it is only a perturbation  of the case 
$\Omega = \Omega_x \times (0,+\infty)$.

\subsection{Improved Boussinesq equation}

Finally, we mention another equation whose null-controllability can be treated with 
our method.
\begin{proposition}
 Let $\Omega = \set R$ or $\set T$. Let $\omega$ be a strict open subset of $\Omega$ 
and let $T>0$. The equation
\begin{equation}\label{eq:impr_boussinesq}
 (\partial_t^2 -\partial_x^2 - \partial_x^2\partial_t^2)f(t,x) = \mathds 1_\omega 
u(t,x), \quad t\in[0,T], x\in \Omega
\end{equation}
is not null-controllable on $\omega$ in time $T$.
\end{proposition}

This equation has been studied by Cerpa and Crépeau~\cite{cerpa_2018}, where it is 
called «improved Boussinesq equation». They prove that, when posed on $\Omega = 
(0,1)$, it is not null-controllable with \emph{boundary control} at $x = 1$. They 
also prove that if $\Omega = \set T$, it is is null-controllable with \emph{moving} 
internal control on $\omega+ct$\footnote{In other words, the right-hand side is 
$\mathds 1_\omega(x-ct) u(t,x)$ instead of $\mathds 1_\omega(x) u(t,x)$.} if the 
speed $c$ is large enough. But while their results suggest the improved Boussinesq 
equation is not null-controllable with (non-moving) internal control, they do not 
prove it. Here, we provide a proof of this fact.

\begin{proof}[Sketch of the proof]
 Let $\xi_0>0$ and $\chi\in C_c^\infty$ to be chosen later. For $\xi\in \set R$ we 
define $\lambda_\xi = \xi^2(1+\xi^2)^{-1}$. For $h>0$, we consider
 \[
  g_h(t,x) = h\int_{\set R}\chi(h\xi-\xi_0)e^{-(h\xi-\xi_0)^2\!/2h +ix\xi -it 
\sqrt{\lambda_\xi}}\diff \xi.
 \]
 Elementary computations prove that $g_h$ is solution of $(\partial_t^2 -\partial_x^2 
-\partial_x^2\partial_t^2)g_h(t,x) = 0$ (it is related to the fact that this equation 
can be rewritten as $(\partial_t^2 -(I-\partial_x^2)^{-1}\partial_x^2)g_h(t,x) = 0$, 
and to the spectral analysis of this operator, see~\cite{cerpa_2018}).

With the notation of Eq.~\eqref{eq:defI}, $g_h(t,x) = I_{t,h,1}(x)$ with $\rho$ 
independant of $(t,h)$ defined by $\rho(\xi) = it\sqrt{\lambda_\xi}$. We can choose 
$K>0$, $\xi_0>0$, and $\chi\in C_c^\infty$ such that \cref{hyp} holds with $X = [0,T]$.

Then, with \cref{th:ighv_asy} and \cref{th:ighv_upper}, we prove that $(g_h)_{h>0}$ 
is a counterexample to the observability inequality associated to the control 
problem~\eqref{eq:impr_boussinesq} in the case $\Omega = \set R$.
In the case $\Omega = \set T$, we periodize $g_h$ as in \cref{sec:heat_bounded}.
\end{proof}

\section{Precise estimation of the eigenfunctions}\label{sec:agmon_lower}
To prove \cref{th:lower_agmon}, we will need the following theorem, which 
is a special case\footnote{In the reference, the Theorem is stated with an open 
(bounded star-shaped) domain $U$ instead of a arbitrary (bounded star-shaped) subset 
$E$ of $\set C$, but we can set $U = E^{\delta}$, and apply the Theorem as stated 
in the reference to get $|H_\gamma(f)|_{L^\infty(E)}\leq 
C_\delta|f|_{L^\infty(E^{2\delta})}$.} of Theorem~18 in~\cite{koenig_2017}.

\begin{theorem}\label{th:est_default}
Let $\mathcal S$ be the space of holomorphic function on the domain $\Omega = 
\{\Re(z)>1\}$ with sub-exponential growth at infinity, i.e.\ $\gamma\in\mathcal S$ 
if and only if for all $\epsilon>0$, $p_\epsilon(\gamma) = 
\sup_{\Re(z)>1}|\gamma(z)e^{-\epsilon|z|}|<+\infty$. We endow $\mathcal S$ with the 
seminorms family $(p_\epsilon)_{\epsilon>0}$. 

Let $\gamma$ in $\mathcal S$ and let $H_\gamma$ be the operator on polynomials with 
a double root at zero, defined by: 
\[H_\gamma\bigg(\sum_{n>1} a_n z^n\bigg) = \sum_{n>1} \gamma(n) a_n z^n.\]
Let $E$ be an bounded subset of $\set C$, star shaped with respect to $0$. Let 
$U$ be a neighborhood of $\bar{E}$. Then there exists 
$C>0$ such that for all polynomials $f$ with a double root at $0$:
\begin{equation}\label{eq:default}
|H_\gamma(f)|_{L^\infty(E)}\leq C|f|_{L^\infty(U)}.
\end{equation}

Moreover, the constant $C$ above can be chosen continuously in $\gamma\in \mathcal 
S$.
\end{theorem}

Note that according to the estimate of the previous 
\cref{th:est_default}, and assuming $U$ is star-shaped with respect to $0$, the 
operators $H_\gamma$ extend by density to every holomorphic 
function\footnote{According to Runge's theorem~\cite[Theorem 13.9]{rudin_1986}, the 
polynomials are dense in the space of holomorphic functions on $U$ with the topology 
of the convergence on every compact.} on $U$ (with a double zero at $0$).
So, we will apply this estimate~\eqref{eq:default} on entire functions (with a double 
zero at $0$).

\begin{proof}[Proof of \cref{th:lower_agmon}]
The proof is made by writing $\tilde u_{\tilde \xi}(v)$ as the power series $\tilde 
u_{\tilde \xi}(v) = \sum \tilde u_{\tilde \xi,2n} v^{2n}$, and showing that the 
coefficients $\tilde u_{\tilde \xi,n}$ of this power 
series are of the form $\tilde u_{\tilde \xi, 2n} = \tilde \rho_{\tilde 
\xi}\gamma_{\tilde \xi}(n) \tilde \xi^n/n!$ for $n\geq 1$, with $\tilde 
\rho_{\tilde \xi}$ 
defined at the beginning of \cref{sec:kolm_bounded}, so that with the 
notation of \cref{th:est_default}: 
\begin{equation}\label{eq:delta_tilde_u}
\tilde u_{\tilde  \xi}(v) = 1 + 
\tilde\rho_{\tilde \xi} H_{\gamma_{\tilde \xi}}(e^{\tilde \xi v^2}-1)(v)
\end{equation}
Then, \cref{th:est_default} will allow us to conclude.

Let us write $\tilde u_{\tilde \xi}(v) = \sum_{n = 0}^{+\infty} \tilde u_{\tilde 
\xi, n} v^n.$
Since $\tilde u_{\tilde\xi}$ satisfies the Cauchy problem $-\tilde u_{\tilde \xi}'' 
+ 2\tilde \xi v \tilde u_{\tilde \xi}'- \tilde \rho_{\tilde \xi}\tilde u_{\tilde 
\xi} = 0$ with initial conditions\footnote{Here we use the fact that $\tilde 
u_{\tilde \xi}$ is even when $\tilde \xi$ is real positive, which is well-known from 
Sturm-Liouville's theory.} $\tilde u_{\tilde \xi} (0) = 1$, $\tilde u_{\tilde 
\xi}'(0) = 0$, we have $\tilde u_{\tilde \xi, 0} = 1$, $\tilde u_{\tilde \xi, 2n+1} 
= 0$ and
\begin{equation}
 \tilde u_{\tilde \xi, n+2} = \frac{2n\tilde \xi-\tilde \rho_{\tilde 
\xi}}{(n+1)(n+2)} \tilde u_{\tilde \xi, n}
\end{equation}
so, by induction, for $n\geq1$
\begin{equation}
 \tilde u_{\tilde \xi, 2n} = -\frac{\tilde \rho_{\tilde \xi}}{2} 
 \frac{(4\tilde \xi)^{n-1} (n-1)!}{(2n)!} 
\prod_{k=1}^{n-1} \left(1-\frac{\tilde \rho_{\tilde \xi}}{4\tilde\xi k}\right).
\end{equation}

So, by defining
\begin{equation}
 \gamma_{\tilde \xi} (n) = -\frac{1}{8\tilde \xi n}\times\frac{4^n(n!)^2}{(2n)!}
\times\prod_{k=1}^{n-1} \left(1-\frac{\tilde \rho_{\tilde \xi}}{4\tilde\xi k}\right)
\end{equation}
we have $\tilde u_{\tilde \xi,2n} = \tilde \rho_{\tilde \xi}\gamma_{\tilde 
\xi}(n) \tilde \xi^n/n!$. So,
$ \tilde u_{\tilde \xi}(v) = 1 +\tilde\rho_{\tilde \xi}
 \sum_{n\geq 1} \gamma_{\tilde \xi}(n) \frac1{n!}{(v^2 \tilde \xi)^n}$.
Assuming that $\gamma_{\tilde \xi}$ is in $\mathcal S$, this is exactly the 
equation~\eqref{eq:delta_tilde_u} we were claiming.

Well, let us actually prove that $\gamma_{\tilde \xi}$ is in the space $\mathcal S$ 
defined 
in \cref{th:est_default}, i.e. that we can extend $n\mapsto \gamma_{\tilde 
\xi}(n)$ to a holomorphic function on $\Omega = \{\Re(z)>1\}$ with 
subexponential growth. This is 
obvious for the term $-1/(8\tilde\xi n)$. The term $4^n(n!)^2/(2n)!$ can be extended 
to $\Omega$ with Euler's Gamma function, and Stirling's approximation gives us the 
subexponential growth (actually an equivalent in $\sqrt{\pi z}$). The product term is a tiny 
bit more tricky to extend to non-integer values. We define it with the following 
formula, which is inspired by~\cite{muller_2005}, and where we have set $\alpha 
= -\tilde \rho_{\tilde \xi}/{4\tilde \xi}$:
\begin{equation}
 \delta_{\tilde \xi}(z) = 
\prod\limits_{k=1}^{+\infty}\dfrac{1+\frac\alpha{k}}{1+\frac\alpha { k+z-1 } }.
\end{equation}

This product converges if $|\alpha|<1/2$ and $\Re(z)>1$. And if $n$ is integer, 
$\delta_{\tilde \xi}(n)$ is a telescopic product, and we have $\delta_{\tilde 
\xi}(n) 
= \prod_{k=1}^{n-1} \left(1+\frac{\alpha}k\right)$. Moreover, $\delta_{\tilde \xi}$ 
is holomorphic on $\Omega$. We also claim that there exists $c,C>0$ 
such that if $|\alpha|<1/2$ and 
$\Re(z)>1$, $|\delta_{\tilde \xi}(z)| \leq C|z|^{c}$. 
The proof of this claim is 
just a few basic computations, and we postpone it after the end of the proof at hand.

Since $\alpha = \tilde \rho_{\tilde \xi}/4\tilde \xi$, 
according to \cref{th:asymptotic}, $|\alpha|<1/2$ as soon as 
$\lvert\arg(\tilde \xi)\rvert < \theta$ and $|\tilde \xi|$ is large enough, say 
$|\tilde \xi|>M$ (depending on $\theta$). Then, according to the claim, the term 
$\delta_{\tilde \xi}(z)$ has subexponential growth 
in $\Omega$, and since it is holomorphic, it is in $\mathcal S$. Moreover, this 
estimate also 
proves that $(\delta_{\tilde \xi})_{|\alpha|<1/2}$ is a 
bounded family of $\mathcal S$.

So $(\gamma_{\tilde \xi})$ is a bounded family of $\mathcal S$ for 
$\lvert\arg(\tilde 
\xi)\rvert<\theta$ and $|\tilde \xi|>M$. So, according to \cref{th:est_default} and 
the following remark, for any 
neighborhood $U$ of
$[-1+\epsilon,1-\epsilon]$ that is star-shaped with respect to $0$, there exists 
$C>0$ such that for all $|\tilde \xi|>M$ with $\lvert\arg(\tilde\xi)\rvert<\theta$ 
and for every $v\in 
(-1+\epsilon,1-\epsilon)$:
\begin{equation}
 \left\lvert H_{\gamma_{\tilde \xi}}(e^{\tilde \xi v^2}-1)(v) \right\rvert \leq
 C(1+|e^{\tilde \xi v^2}|_{L^\infty(U)})
\end{equation}
and if we choose $U$ to be small enough, we have $\lvert H_{\gamma_{\tilde 
\xi}}(e^{\tilde \xi v^2}-1)(v) \rvert
\leq C' \lvert e^{(1-\delta)\tilde \xi}\rvert.$
Finally, thanks to equation~\eqref{eq:delta_tilde_u} and \cref{th:asymptotic}, we 
have
\begin{equation}
 \lvert\tilde u_{\tilde \xi}(v) -1\rvert
 \leq C_\delta|\tilde\xi|^{3/2}|e^{-\delta\tilde \xi}|.
\end{equation}
\end{proof}
\begin{proof}[Proof of the claim that $|\delta_{\tilde \xi}(z)|\leq C|z|^c$]
We first write
\begin{equation}\label{eq:def_delta_symbol}
 \delta_{\tilde \xi}(z) = \exp\left(
\sum_{k=1}^{+\infty}\ln\left(1+\frac\alpha{k}\right) - 
\ln\left(1+\frac\alpha{k+z-1}\right)\right).
\end{equation}
Let us also remind that we assume $|\alpha|<1/2$ and $\Re(z)>1$, so that 
for $k\in \set N^\ast$ $|\alpha/k|<1/2$ and $|\alpha/(k+z-1)|<1/2$.
We denote $k_0 = \lfloor|z|\rfloor$, and we separate the sum into two parts:
\begin{equation*}
 S_{\leq k_0} = \sum_{k=1}^{k_0}\ln\left(1+\frac\alpha{k}\right) - 
\ln\left(1+\frac\alpha{k+z-1}\right)\qquad S_{>k_0} =\!\! 
\sum_{k=k_0+1}^{+\infty}\!\!\ln\left(1+\frac\alpha{k}\right) - 
\ln\left(1+\frac\alpha{k+z-1}\right)
\end{equation*}

About the part of a sum for $k\leq 
k_0$, we have thanks to the triangle inequality and the fact that for 
$|x|<1/2$, $\lvert\ln(1+x)\rvert\leq c|x|$:
\begin{equation}
\left|S_{\leq k_0}\right|
\leq 2c|\alpha|\sum_{k=1}^{k_0} \frac1k
\leq 2c|\alpha|(\ln(k_0) + C')
\leq 2c|\alpha|(\ln(|z|) +C')\label{eq:claim_small},
\end{equation}
where we used the relation between the harmonic sum and the logarithm and 
the fact that $k_0 = \lfloor|z|\rfloor$.
About the rest of the sum, we have by writing $\ln(1+b)-\ln(1+a) = \int_a^b 
\frac{\diff[0]x}{1+x}$,
\begin{equation}
\left|S_{>k_0}\right|
\leq\sum_{k=k_0+1}^{+\infty} 
\left|\int_{\alpha/k}^{\alpha/(k+z-1)}\frac{\diff[0]x}{1+x}\right| \leq 
\sum_{k=k_0+1}^{+\infty} 
2\left|\frac{\alpha}{k}-\frac{\alpha}{k+z-1}\right|
\leq2|\alpha (z-1)| \sum_{k=k_0+1}^{+\infty} \frac1{k^2},
\end{equation}
where we used the fact that for 
$x\in[\frac\alpha k,\frac\alpha{k+z-1}]$, $|\frac1{1+x}|\leq 2$. By comparing 
this sum with an integral,
\begin{equation}
\left|S_{>k_0}\right|
\leq 2|\alpha (z-1)|\int_{k_0}^{+\infty}\frac{\diff[0]x}{x^2}
\leq 2|\alpha| \frac{|z-1|}{k_0} \leq C''|\alpha|,\label{eq:claim_infinity}
\end{equation}
where we again used that $k_0 = \lfloor|z|\rfloor$. Summing the two 
inequalities~\eqref{eq:claim_small} and~\eqref{eq:claim_infinity}, and plugging this 
into equation~\eqref{eq:def_delta_symbol} proves the claim.
%
\end{proof}

\section*{Acknowledgments} The author would like to thank his Ph.D. advisor, 
Gilles Lebeau, for the counterexample to the observability inequality of the 
fractional heat equation and several other enlightening discussions. He also thanks 
Karine Beauchard for many interesting discussions.

\bibliographystyle{siamplain}
\bibliography{references}
\end{document}